\documentclass[12pt,a4paper]{article}
\usepackage{geometry}
\usepackage{times,authblk}
\usepackage[UKenglish]{babel}
\usepackage{amsmath, amsfonts, amssymb, amsthm, bm, stmaryrd, enumerate, mathtools}
\usepackage{graphicx, tikz-cd}
\renewcommand{\epsilon}{\varepsilon}
\renewcommand{\phi}{\varphi}
\renewcommand{\rho}{\varrho}

\newtheorem{Def}{Definition}[section]
\newenvironment{definition}{\begin{Def} \rm}{\end{Def}}
\newtheorem{lemma}[Def]{Lemma}
\newtheorem{proposition}[Def]{Proposition}

\newtheorem{theorem}[Def]{Theorem}
\newtheorem{example}[Def]{Example}

\newcommand{\komma}{,\hspace{0.3em}}
\newcommand{\id}{\text{id}}
\renewcommand{\leq}{\leqslant}
\renewcommand{\geq}{\geqslant}
\renewcommand{\emptyset}{\varnothing}

\newcommand{\Naturals}{{\mathbb N}}

\newcommand{\Rationals}{{\mathbb Q}}
\newcommand{\Reals}{{\mathbb R}}
\newcommand{\Complexes}{{\mathbb C}}

\newcommand{\NOS}{{\mbox{${\mathcal N}{\mathcal O}{\mathcal S}$}}}
\newcommand{\LOS}{{\mbox{${\mathcal L}{\mathcal O}{\mathcal S}$}}}
\newcommand{\EOS}{{\mbox{${\mathcal E}{\mathcal O}{\mathcal S}$}}}

\newcommand{\notperp}{\mathbin{\not\perp}}
\newcommand{\perpe}[1]{\mathbin{\perp_{#1}}}
\renewcommand{\c}{^\perp}
\newcommand{\cc}{^{\perp\perp}}
\newcommand{\ce}[1]{^{\perp_{#1}}}
\newcommand{\cce}[1]{^{\perp_{#1}\perp_{#1}}}
\newcommand{\herm}[2]{( #1 , #2 )}
\newcommand{\hermbig}[2]{\left( #1 , #2 \right)}
\newcommand{\lin}[1]{[#1]}
\newcommand{\withoutzero}{^{\raisebox{0.2ex}{\scalebox{0.4}{$\bullet$}}}}
\newcommand{\msg}{^{\raisebox{0.2ex}{\scalebox{0.6}{$\times$}}}}
\newcommand{\n}{{\mathbf n}}

\newcommand{\Two}{{\mathbf 2}}
\newcommand{\Three}{{\mathbf 3}}
\renewcommand{\vector}[3]{{\tiny \begin{pmatrix} #1 \\ #2 \\ #3 \end{pmatrix}}}
\DeclarePairedDelimiter{\norm}{\lVert}{\rVert}
\renewcommand{\angle}{\sphericalangle\,}

\begin{document}

\title{Categories of orthogonality spaces}

\author[1]{Jan Paseka}

\author[2]{Thomas Vetterlein}

\affil[1]{\footnotesize
Department of Mathematics and Statistics,
Masaryk University \authorcr
Kotl\'a\v rsk\'a 2, 611\,37 Brno, Czech Republic \authorcr
{\tt paseka@math.muni.cz}}

\affil[2]{\footnotesize
Department of Knowledge-Based Mathematical Systems,
Johannes Kepler University Linz \authorcr
Altenberger Stra\ss{}e 69, 4040 Linz, Austria \authorcr
{\tt Thomas.Vetterlein@jku.at}}

\date{\today}

\maketitle

\begin{abstract}\parindent0pt\parskip1ex

\noindent An orthogonality space is a set equipped with a symmetric and irreflexive binary relation. We consider orthogonality spaces with the additional property that any collection of mutually orthogonal elements gives rise to the structure of a Boolean algebra. Together with the maps that preserve the Boolean structures, we are led to the category \NOS\ of normal orthogonality spaces.

Moreover, an orthogonality space of finite rank is called linear if for any two distinct elements $e$ and $f$ there is a third one $g$ such that exactly one of $f$ and $g$ is orthogonal to $e$ and the pairs $e, f$ and $e, g$ have the same orthogonal complement. Linear orthogonality spaces arise from finite-dimensional Hermitian spaces. We are led to the full subcategory \LOS\ of \NOS\ and we show that the morphisms are the orthogonality-preserving lineations.

Finally, we consider the full subcategory \EOS\ of \LOS\ whose members arise from positive definite Hermitian spaces over Baer ordered $\star$-fields with a Euclidean fixed field. We establish that the morphisms of \EOS\ are induced by generalised semiunitary mappings.

{\it Keywords:} Orthogonality spaces; undirected graphs; categories; Boolean subalgebras; linear orthogonality spaces; generalised semilinear map; generalised semiunitary map

{\it MSC:} 81P10; 06C15; 46C05

\mbox{}\vspace{-2ex}

\end{abstract}

\section{Introduction}
\label{sec:Introduction}

In quantum mechanics, physical processes are described in a way assigning an essential role to the observer. Rather than predicting on the basis of complete initial conditions the unambiguous development of some physical system, the theory assigns probabilities to pairs consisting of a preparation procedure and the outcome of a subsequent measurement. Why the formalism has proved successful is by and large today still unanswered; we could admit that we rather got used to it. But even at the most basic level, there are unresolved issues. A key ingredient of the model is a certain inner-product space -- a Hilbert space over the field of complex numbers --, and the deeper reasons for this choice are a matter of ongoing discussions.

The probably oldest approach aiming to clarify the basic principles on which quantum theory is based is due to Birkhoff and von Neumann \cite{BiNe}. The keyword ``quantum logic'' is often used in this context but might be misleading. What in our eyes rather matters is the idea of increasing the degree of abstraction: the question is whether the Hilbert space can be recovered from a considerably simpler structure. Numerous types of algebras, including partial ones, have been proposed and investigated, the best-known example being orthomodular lattices, which describe the Hilbert space by means of the inner structure of its closed subspaces. For an overview of possible directions, we may refer, e.g., to the handbooks \cite{EGL1,EGL2}.

Increasing the degree of abstraction means to restrict the structure to the necessary minimum. An approach that was proposed in the 1960s by David Foulis and his collaborators goes presumably to the limits of what is possible. They coined the notion of an orthogonality space, which is simply a set endowed with a symmetric and irreflexive binary relation \cite{Dac,RaFo,Wlc}. The prototypical example is the collection of one-dimensional subspaces of a Hilbert space together with the usual orthogonality relation.

The notion of an orthogonality space is in the centre of the present work and the main motivation behind our work is to elaborate on its role within the basic quantum-physical model. We generally deal with the case of a finite rank, meaning that there are only finitely many pairwise orthogonal elements. We should certainly be aware of the fact that orthogonality spaces are as general as undirected graphs, which in turn are rarely put into context with inner-product spaces. As has been shown in \cite{Vet3}, however, the relationship between the two types of structures is close. An orthogonality space of finite rank is called {\it linear} if, for any distinct elements~$e$ and~$f$, there is a further one~$g$ such that exactly one of~$f$ and~$g$ is orthogonal to~$e$ and the set of elements orthogonal to~$e$ and~$f$ coincides with the set of elements orthogonal to~$e$ and~$g$. Linearity characterises the orthogonality spaces that arise from finite-dimensional Hermitian spaces.

In physics, symmetries of the model generally play a fundamental role. It might thus not come as a surprise that orthogonality spaces associated with complex Hilbert spaces are describable by the particular properties of their automorphisms \cite{Vet1,Vet2}. Here, we explore this issue further, but we adopt a more general perspective than in the previous works.

The present paper is devoted to the investigation of structure-preserving maps between orthogonality spaces. We do so first in a general context, taking into account features inherent to orthogonality spaces, and in a second step, we turn to the narrower class of linear orthogonality spaces. We start with the question how to reasonably define morphisms. It certainly seems to make sense to require nothing more than the preservation of the single binary relation on which the structures are based. We call orthogonality-preserving maps homomorphisms. To choose homomorphisms as morphisms, however, is inexpedient when the context that we ultimately have in mind is given by inner-product spaces. Indeed, for linear orthogonality spaces, we expect a morphism to preserve, in some sense, linear dependence. The following situation illustrates the difficulties \cite{Sem}, even though we otherwise deal with the finite-dimensional case only. Consider the complex projective space over three dimensions $P(\Complexes^3)$ as well as over $2^{\aleph_0}$ dimensions $P(\Complexes^{2^{\aleph_0}})$; then any injective map from $P(\Complexes^3)$ to $P(\Complexes^{2^{\aleph_0}})$ such that the image consists of mutually orthogonal elements is a homomorphism of orthogonality spaces, but in no way related to the preservation of linear dependence.

We note that these problems do not arise in approaches that consider the orthogonality relation not as basic but as an additional structure. Projective geometries enhanced by an orthogonality relation were studied, e.g., in \cite{FaFr,StSt}. Here, we try an alternative way. Having in mind the Hilbert space model of quantum physics, we focus on an adjusted kind of orthogonality spaces, ruling out structural peculiarities that we must consider as inappropriate. In quantum mechanics, observables correspond to Boolean algebras. In a finite-state system, measurement outcomes correspond to mutually orthogonal subspaces, which in turn generate a Boolean subalgebra of the lattice of closed subspaces. We require to have an analogue of this situation in our more abstract setting and we take it into account for our definition of morphisms.

To be more specific, let us first recall that orthogonality spaces lead us straightforwardly to the realm of lattice theory. A subset $A$ of an orthogonality space $(X, \perp)$ is called orthoclosed if $A = B\c$ for some $B \subseteq X$, where $B\c$ is the set of $e \in X$ orthogonal to all elements of $B$. The set of orthoclosed subsets form a complete ortholattice ${\mathcal C}(X, \perp)$. Now, consider a collection $E = \{x_1, \ldots, x_k\}$ of mutually orthogonal elements of $X$. Then the subsets of $E$ generate a subortholattice of ${\mathcal C}(X, \perp)$. This subortholattice is, in general, not isomorphic to the Boolean algebra of subsets of $E$; in case it always is, we call $(X, \perp)$ {\it normal}. We moreover name homomorphisms in the same way if they preserve, in a natural sense, Boolean subalgebras of ${\mathcal C}(X, \perp)$. We thus arrive at the category \NOS\ of normal orthogonality spaces and normal homomorphisms.

We take up in this way an often-discussed issue. Indeed, for the aim of recovering a Hilbert space or, more generally, an orthomodular lattice from suitable substructures, it has been a guiding motive to consider the lattice as being glued together from its Boolean subalgebras; see, e.g., \cite[Section 4]{Nav}. Moreover, deep results have been achieved on the question how to reconstruct orthomodular lattices or related quantum structures from the poset of their Boolean subalgebras \cite{HaNa,HHLN}.

Any linear orthogonality space is normal and thus our next step is to consider normal homomorphisms between linear orthogonality spaces. That is, we investigate the full subcategory \LOS\ of \NOS, consisting of linear orthogonality spaces. It turns out that the morphisms in \LOS\ do have the most basic property to be expected: they are maps between projective spaces that preserve the triple relation of being contained in a line, that is, they are lineations. In fact, we show that the morphisms are exactly the orthogonality-preserving lineations.

Our final objective is to describe the morphisms in \LOS\ as precisely as possible. Generalisations of the fundamental theorem of projective geometry show that any lineation is induced by a generalised semilinear transformation -- provides it is non-degenerate \cite{Mach,Fau}. Here, non-degeneracy means two additional conditions to hold: (1) the image is not contained in a $2$-dimensional subspace, and (2) the image of a line is never $2$-element. Provided that the rank is at least $3$, condition (1) is ensured. Condition (2), however, leads us to an issue dealt with in the discussions around the peculiarities of quantum physics: we show that a violation of (2) implies the existence of two-valued measures. The exclusion of two-valued measures is in turn a consequence of Gleason's Theorem in case that the skew field is $\Complexes$ or $\Reals$. Although the case of specific further skew fields has been discussed \cite{Dvu}, not much seems to be known about the general case. Here, we show that if the skew field of scalars is a Euclidean subfield of the reals, two-valued measures do not exist. Consequently, the same applies if a $\star$-field possesses a subfield of this type.

Moreover, we deal with lineations that in addition preserve an orthogonality relation. It seems natural to ask whether the representing generalised semilinear map can be chosen to preserve in some sense the inner product. We establish that this is the case under particular conditions: the skew field is commutative, that is, a field, and there is a basis of vectors of equal length. These conditions apply for positive definite Hermitian spaces over Baer ordered $\star$-fields whose fixed field is Euclidean. In this case, morphisms are induced by what we call generalised semiunitary maps.

The paper is organised as follows. In the following Section~\ref{sec:Normal-orthogonality-spaces}, we fix the basic notation used in this paper. Moreover, we introduce and discuss normal orthogonality spaces, including a characterisation of normality as an intrinsic property, without reference to the associated ortholattice. In Section~\ref{sec:NOS}, we investigate the category \NOS\ of normal orthogonality spaces and normal homomorphisms. In Section~\ref{sec:Hermitian-spaces}, we prepare the ground for the discussion of those orthogonality spaces that arise from inner-product spaces; in particular, we discuss lineations between projective spaces and discuss their representation in the presence of an inner product. Then, in Section~\ref{sec:Linear-orthogonality-spaces}, we recall the notion of a linear orthogonality space and we show that linearity implies normality. Finally, in Section~\ref{sec:LOS}, we study the full subcategory \LOS\ of \NOS\ that consists of linear orthogonality spaces, focussing especially on the morphisms induced by generalised semiunitary maps. Some concluding remarks are found in the final Section~\ref{sec:Conclusion}.

\section{Normal orthogonality spaces}
\label{sec:Normal-orthogonality-spaces}

We deal in this paper with the following relational structures.

\begin{definition}
An {\it orthogonality space} is a non-empty set $X$ equipped with a symmetric, irreflexive binary relation $\perp$, called the {\it orthogonality relation}. The supremum of the cardinalities of sets of mutually orthogonal elements of $X$ is called the {\it rank} of~$(X, \perp)$.
\end{definition}

We may observe that orthogonality spaces are essentially the same as undirected graphs, understood such that the edges are $2$-elements subsets of the set of nodes. The rank of an orthogonality space is under this identification the supremum of the sizes of cliques. The present work, however, is not motivated by graph theory, our guiding example rather originates in quantum physics.

\begin{example} \label{ex:standard-example-1}
Let $H$ be a Hilbert space. Then the set $P(H)$ of one-dimensional subspaces of $H$, together with the usual orthogonality relation, is an orthogonality space, whose rank coincides with the dimension of $H$.
\end{example}

The {\it (orthogonal) complement} of a subset $A$ of an orthogonality space $X$ is
\[ A\c \;=\; \{ x \in X \colon x \perp a \text{ for all $a \in A$} \}. \]
The map ${\mathcal P}(X) \to {\mathcal P}(X) \komma A \mapsto A\cc$ is a closure operator on $X$. We call the closed subsets {\it orthoclosed} and we denote the collection of orthoclosed subsets by ${\mathcal C}(X, \perp)$. Endowed with the set-theoretical inclusion and the orthocomplementation $\c$, ${\mathcal C}(X, \perp)$ becomes a complete ortholattice. The ortholattice $({\mathcal C}(X, \perp); \cap, \vee, ^\perp, \emptyset, X)$ will be our primary tool to investigate $(X, \perp)$.

\begin{example} \label{ex:standard-example-2}
Let $(P(H), \perp)$ be the orthogonality space arising from the Hilbert space $H$ according to Example~\ref{ex:standard-example-1}. Then we may identify ${\mathcal C}(P(H), \perp)$ with the set ${\mathcal C}(H)$ of closed subspaces of $H$, endowed with the set-theoretical inclusion and the orthocomplementation.
\end{example}

In this paper, we will focus exclusively on the case of a finite rank. Our guiding example is, accordingly, the orthogonality space associated with a finite-dimensional Hilbert space. From now on, all orthogonality spaces are tacitly assumed to be of finite rank.

We will next introduce a condition on orthogonality spaces that mimics a key feature of the quantum-physical formalism. In quantum mechanics, a physical system is modelled by means of a Hilbert space and observables correspond to Boolean subalgebras of the lattice of its closed subspaces. We will require that orthogonality spaces possess substructures of the corresponding type.

\begin{definition} \label{def:normal}
An orthogonality space $(X, \perp)$ is called {\it normal} if, for any mutually orthogonal elements $e_1, \ldots, e_k$ of $X$, where $k \geq 1$, the subalgebra of the ortholattice ${\mathcal C}(X, \perp)$ generated by $\{e_1\}\cc, \ldots,$ $\{e_k\}\cc$ is Boolean.
\end{definition}

We may understand normality also as a coherence condition. By a subset $A$ of an orthogonality space to be orthogonal, we mean that $A$ consists of mutually orthogonal elements.

\begin{lemma} \label{lem:normal-orthogonality-space}
For an orthogonality space $(X, \perp)$, the following are equivalent:
\begin{enumerate}[{\rm(1)}]

\item $(X, \perp)$ is normal.

\item For any maximal orthogonal set $\{ e_1, \ldots, e_n \} \subseteq X$, there is a finite Boolean subalgebra of ${\mathcal C}(X, \perp)$ whose atoms are $\{ e_1 \}\cc, \ldots, \{ e_n \}\cc$.

\item For any maximal orthogonal set $\{ e_1, \ldots, e_n \} \subseteq X$ and any $1 \leq k < n$, if $f \perp e_1, \ldots, e_k$ and $g \perp e_{k+1}, \ldots, e_n$, then $f \perp g$.

\end{enumerate}
\end{lemma}

\begin{proof}
(1) $\Rightarrow$ (2): Let $(X, \perp)$ be normal and let $\{ e_1, \ldots, e_n \}$ be a maximal orthogonal subset of $X$. By normality, the subalgebra $\mathcal B$ of ${\mathcal C}(X, \perp)$ generated by $\{ e_1 \}\cc, \ldots,$ $\{ e_n \}\cc$ is Boolean. Moreover, $\{ e_1 \}\cc, \ldots, \{ e_n \}\cc$ are mutually orthogonal elements and we have $\{ e_1 \}\cc \vee \ldots \vee \{ e_n \}\cc = \{ e_1, \ldots, e_n \}\cc = \emptyset\c = X$. Thus $\mathcal B$ is a finite Boolean subalgebra of ${\mathcal C}(X, \perp)$, its atoms being $\{ e_1 \}\cc, \ldots,$ $\{ e_n \}\cc$.

(2) $\Rightarrow$ (3): Let $\{ e_1, \ldots, e_n \}$ be a maximal orthogonal subset of $X$ and assume that $\{ e_1 \}\cc, \ldots, \{ e_n \}\cc$ are the atoms of a finite Boolean subalgebra of ${\mathcal C}(X, \perp)$. Let $1 \leq k < n$. Then $f \perp e_1, \ldots, e_k$ means $f \in \{ e_1, \ldots, e_k \}\c = \{ e_{k+1}, \ldots, e_n \}\cc$, and similarly, $g \perp e_{k+1}, \ldots, e_n$ means $g \in \{ e_1, \ldots, e_k \}\cc$. If both $f \perp e_1, \ldots, e_k$ and $g \perp e_{k+1}, \ldots, e_n$ holds, we hence conclude $f \perp g$.

(3) $\Rightarrow$ (1): Let $D = \{ e_1, \ldots, e_k \}$, $k \geq 1$, be an orthogonal subset of $X$. Then we may extend $D$ to a maximal orthogonal subset $E = \{e_1, \ldots, e_n \}$ of $X$, where $n \geq k$. For any $A \subseteq E$, we have $\bigvee \{ \{ e \}\cc \colon e \in A \} = A\cc$; for any $A, B \subseteq E$, we have $A\cc \vee B\cc = (A \cup B)\cc$; and $E\cc = X$. Let $\emptyset \neq A \subsetneq E$. Then $(E \setminus A)\cc \subseteq A\c$. Moreover, if $f \in A\c$ and $g \in (E \setminus A)\c$, we have by assumption $f \perp g$; hence $f \in (E \setminus A)\cc$. We conclude that $A\cc = (E \setminus A)\c$. We have shown that $\{ e_1 \}\cc, \ldots, \{ e_n \}\cc$ generate a Boolean subalgebra of ${\mathcal C}(X, \perp)$; hence so do $\{ e_1 \}\cc, \ldots, \{ e_k \}\cc$.
\end{proof}

The following notation will be useful. Let $e_1, \ldots, e_k$ be mutually orthogonal elements of a normal orthogonality space $(X, \perp)$. Then the closure of $\{ \{e_1\}\cc, \ldots,$ $\{e_k\}\cc \}$ under joins in ${\mathcal C}(X, \perp)$ has the structure of a Boolean algebra, whose top element is $\{ e_1, \ldots, e_k \}\cc$. We will denote this Boolean algebra by ${\mathcal B}(e_1, \ldots, e_k)$.

The property of normality applies to our canonical example. We write $\lin{x_1, \ldots, x_k}$ for the linear hull of non-zero vectors $x_1, \ldots, x_k$ of a linear space.

\begin{example}
Let $x_1, \ldots, x_k$, $k \geq 1$, be mutually orthogonal non-zero vectors of a Hilbert space $H$. Then the subalgebra of ${\mathcal C}(H)$ generated by $\lin{x_1}, \ldots, \lin{x_k}$ consists of the joins of subspaces among $\lin{x_1}, \ldots, \lin{x_k}, \lin{x_1, \ldots, x_k}\c$. This algebra is Boolean and we conclude that $(P(H), \perp)$ is normal.
\end{example}

For later considerations, we introduce a further, particularly simple example.

\begin{example} \label{ex:n-element-set}
For $n \in \Naturals \setminus \{0\}$, we denote by $\n$ an $n$-element set and we consider the binary relation $\neq$ on $\n$. Then $(\n, \neq)$ is an orthogonality space and ${\mathcal C}(\n, \neq)$ is the powerset of $\n$. Since ${\mathcal C}(\n, \neq)$ is Boolean, we have that $(\n, \neq)$ is normal.
\end{example}

In general, however, an orthogonality space need not be normal. The subsequent examples of finite orthogonality spaces will be graphically depicted as follows: the elements of the space are represented by points, and two elements are orthogonal if the points are connected by a straight line. For instance, in the Example~\ref{ex:Finite-example-1} below we have that $a, b, c$ are mutually orthogonal and moreover $d \perp a$ as well as $e \perp b, c$. We note that this representation might remind of Greechie diagrams. It must be kept in mind, however, that an element of an orthogonality space does not necessarily represent an atom of the associated ortholattice. In Example~\ref{ex:Finite-example-1}, for instance, $\{e\}\cc$ properly contains $\{a\}\cc$.

\begin{example} \label{ex:Finite-example-1}
Consider the orthogonality space $X = \{ a, b, c, d, e \}$ given by the following scheme:

\begin{center}
\includegraphics[width=0.15\textwidth]{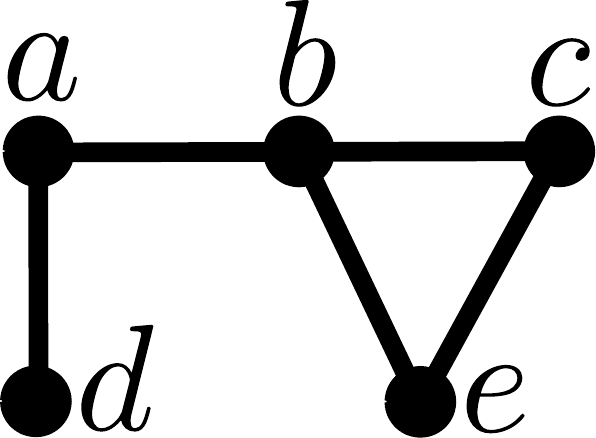}
\end{center}

$\{ a, b, c \}$ is a maximal orthogonal set. Furthermore, we have $\{ a \}^{\perp\perp} = \{ a \}$ and $\{ b \}\cc \vee \{ c \}\cc = \{ b, c \}\cc = \{ b, c \}$. Since $\{b, c\}\c = \{ a, e \}$, $X$ is not normal.
\end{example}

Given a normal orthogonality space $(X, \perp)$, we call an orthoclosed subset $A$ of $X$ together with the inherited orthogonality relation, which we usually still denote by $\perp$, a {\it subspace} of $(X, \perp)$.

The following proposition and example show that a subspace of a normal orthogonality space is not in general normal, but a subspace that is the closure of any maximal orthogonal subset is so.

\begin{proposition} \label{prop:normal-subspaces}
Let $(X, \perp)$ be a normal orthogonality space and let $A \in {\mathcal C}(X, \perp)$ be such that, for any maximal orthogonal subset $D$ of $A$, we have $D\cc = A$. Then the subspace $(A, \perp)$ is normal.
\end{proposition}

\begin{proof}
We shall use criterion (3) of Lemma~\ref{lem:normal-orthogonality-space}. Let $\{ e_1, \ldots, e_n \}$ be a maximal orthogonal subset of $A$, let $1 \leq k < n$, and assume that there are $f, g \in A$ such that $f \perp e_1, \ldots, e_k$ and $g \perp e_{k+1}, \ldots, e_n$. Then we may choose $e_{n+1}, \ldots, e_m \in X$ such that $\{ e_1, \ldots, e_m\}$ is a maximal orthogonal subset of $X$. By assumption, $A = \{ e_1, \ldots, e_n\}\cc$, hence $g \perp e_{n+1}, \ldots, e_m$. Thus we have $f \perp e_1, \ldots, e_k$ and $g \perp e_{k+1}, \ldots, e_m$ and the normality of $X$ implies $f \perp g$. We conclude that $(A, \perp)$ is normal.
\end{proof}

\begin{example} \label{ex:Finite-example-3}
Let $(X, \perp)$ be the $14$-element orthogonality space given as follows:

\begin{center}
\includegraphics[width=0.38\textwidth]{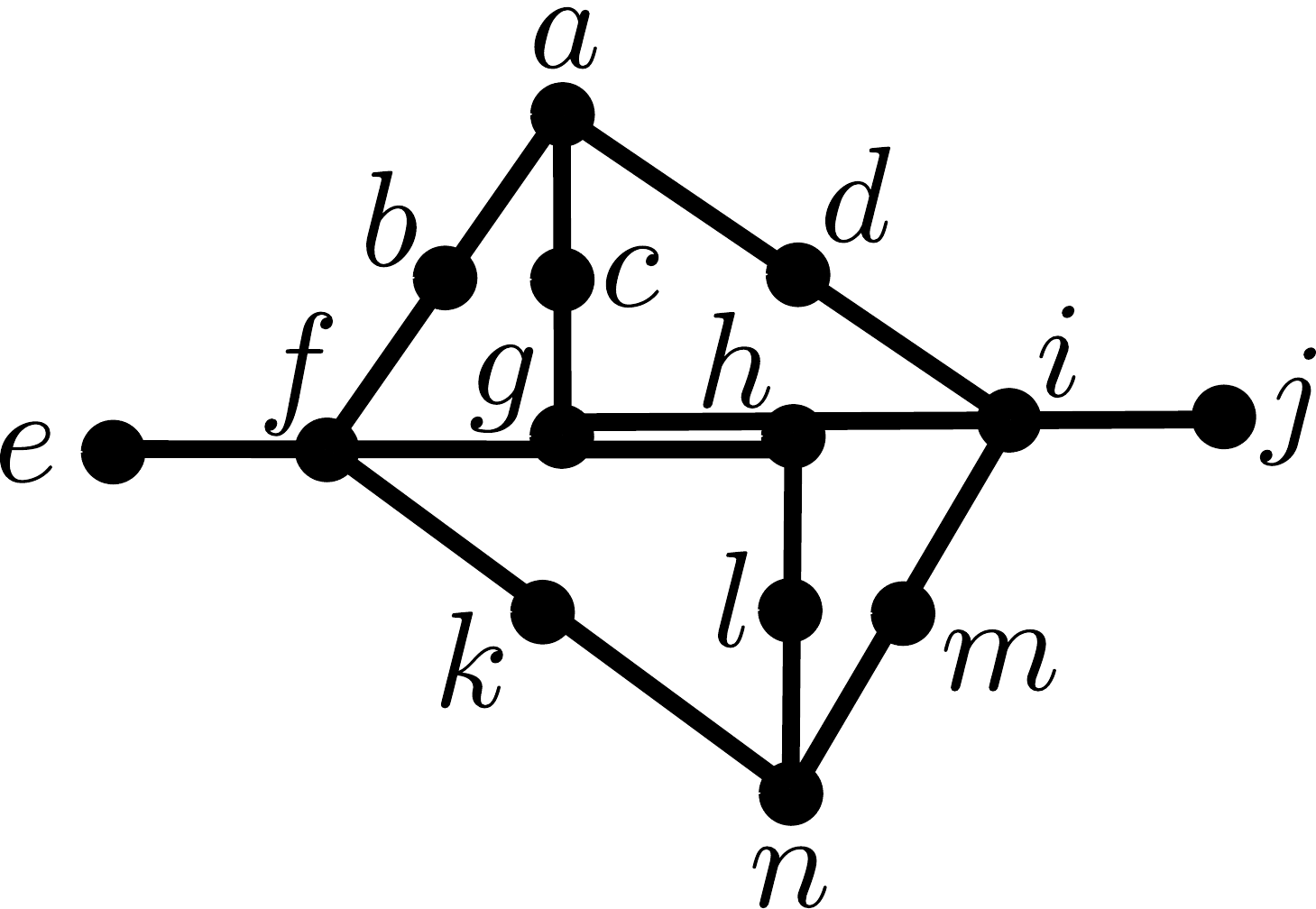}
\end{center}

(Here, the sets $\{ e, f, g, h \}$ and $\{ g, h, i, j \}$ are meant to be orthogonal; but, none of $e$ or $f$ is orthogonal to $i$ or $j$.)

By criterion {\rm (3)} of Lemma~\ref{lem:normal-orthogonality-space}, we may check that $X$ is normal. However, the subspace $\{f,i\}\c = \{a, g, h, n\}$ is not.
\end{example}

We might expect that normality of an orthogonality space is closely related to the orthomodularity of the associated ortholattice. This is indeed the case but the two properties do not coincide.

A {\it Dacey space} is an orthogonality space $(X, \perp)$ such that ${\mathcal C}(X, \perp)$ is an orthomodular lattice. We have the following characterisation of Dacey spaces \cite{Dac,Wlc}.

\begin{lemma} \label{lem:Dacey}
An orthogonality space $(X, \perp)$ is a Dacey space if and only if, for any $A \in {\mathcal C}(X, \perp)$ and any maximal orthogonal subset $D$ of $A$, we have that $D\cc = A$.
\end{lemma}

\begin{example}
Let $H$ be a Hilbert space. Then ${\mathcal C}(H)$ is an orthomodular lattice and hence $(P(H), \perp)$ a Dacey space.
\end{example}

\begin{example}
By means of Lemma~\ref{lem:Dacey}, we observe that the orthogonality space $(X, \perp)$ from Example~\ref{ex:Finite-example-1} is not a Dacey space. Indeed, $A = \{ b, c, d \} \in {\mathcal C}(X, \perp)$, $\{ b, c \}$ is a maximal orthogonal subset of $A$, and $\{ b, c \}\cc = \{ b, c \} \subsetneq A$.
\end{example}

The following proposition and example show that the Dacey spaces form a strict subclass of the normal orthogonality spaces.

\begin{proposition} \label{prop:if-Dacey-then-normal}
A Dacey space is a normal orthogonality space.
\end{proposition}

\begin{proof}
Let $(X, \perp)$ be a Dacey space and $\{ e_1, \ldots, e_k \}$ be an orthogonal subset of $X$. Then $\{ e_i \}\cc$, $i = 1, \ldots, k$, are pairwise orthogonal and hence pairwise commuting elements of the orthomodular lattice ${\mathcal C}(X, \perp)$. It follows that they generate a Boolean subalgebra \cite[Prop.~2.8]{BrHa}.
\end{proof}

\begin{example} \label{Finite-example-2}
Consider the following orthogonality space $(X, \perp)$:

\begin{center}
\includegraphics[width=0.22\textwidth]{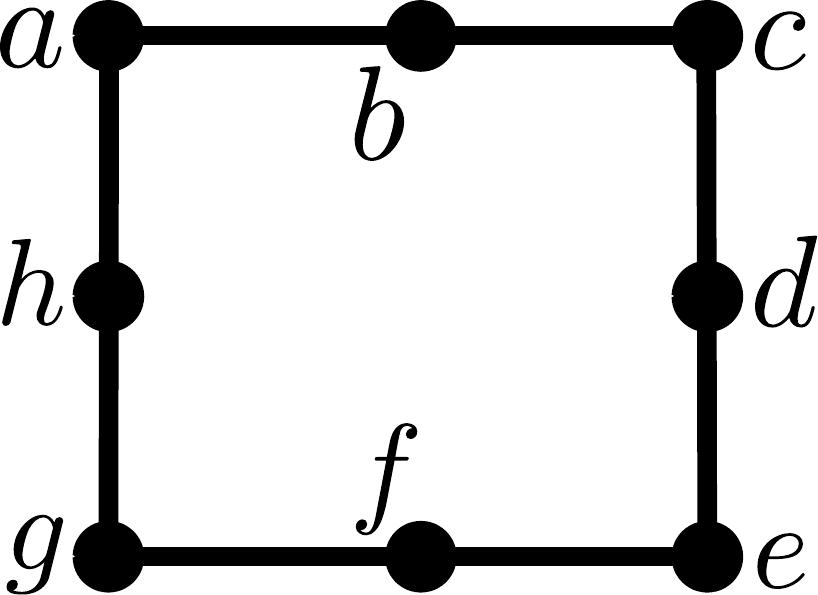}
\end{center}

The maximal orthogonal subsets are the elements along a straight line, e.g., $\{ a, b, c \}$. By criterion {\rm (3)} of Lemma~\ref{lem:normal-orthogonality-space}, we observe that $(X, \perp)$ is normal. We may also check that each subspace of $(X, \perp)$ is normal.

Moreover, the set $\{ a, e \}$ is orthoclosed. But $\{ a \}$ is a maximal orthogonal subset of $\{ a, e \}$ and $\{ a \}\cc = \{ a \}$. Hence by Lemma~\ref{lem:Dacey}, $(X, \perp)$ is not a Dacey space.
\end{example}

\section{The category \NOS\ of normal orthogonality spaces}
\label{sec:NOS}

We discuss in this section structure-preserving maps between orthogonality spaces. We shall introduce a category consisting of normal orthogonality spaces and investigate its basic properties.

For orthogonality spaces $X$ and $Y$, we call a map $\phi \colon X \to Y$ a {\it homomorphism} if $\phi$ is orthogonality-preserving, that is, if, for any $e, f \in X$, $e \perp f$ implies $\phi(e) \perp \phi(f)$. In this case, $\phi$ induces the map
\[ \bar\phi \colon {\mathcal C}(X, \perp) \to {\mathcal C}(Y, \perp)
\komma A \mapsto \{ \phi(a) \colon a \in A \}\cc. \]
Obviously, $\bar\phi$ is order- and orthogonality-preserving. It seems that in general, however, we cannot say much more about $\bar\phi$. We will be interested in homomorphisms fulfilling the following additional condition.

\begin{definition}
Let $\phi \colon X \to Y$ be a homomorphism between the normal orthogonality spaces $X$ and $Y$. We will call $\phi$ {\it normal} if, for any orthogonal set $e_1, \ldots, e_k \in X$, $k \geq 1$, $\bar\phi$ maps ${\mathcal B}(e_1, \ldots, e_k)$ isomorphically to ${\mathcal B}(\phi(e_1), \ldots, \phi(e_k))$.
\end{definition}

The following lemma might help to elucidate the condition of normality for homomorphisms.

\begin{lemma} \label{lem:normal}
Let $\phi \colon X \to Y$ be a homomorphism between normal orthogonality spaces. Then the following are equivalent:
\begin{enumerate}[{\rm(1)}]

\item $\phi$ is normal.

\item For any orthogonal subset $\{ e_1, \ldots, e_k \}$ of $X$, where $k \geq 0$, we have \linebreak $\bar\phi(\{ e_1, \ldots, e_k \}\cc) = \{ \phi(e_1), \ldots, \phi(e_k) \}\cc$.

\item For any orthogonal subset $\{ e_1, \ldots, e_k \}$ of $X$, where $k \geq 0$, we have \linebreak $\phi(\{ e_1, \ldots, e_k \}\cc) \subseteq \{ \phi(e_1), \ldots, \phi(e_k) \}\cc$.

\item For any maximal orthogonal subset $\{ e_1, \ldots, e_n \}$ of $X$, we have $\phi(X)\cc = \linebreak \{ \phi(e_1), \ldots, \phi(e_n) \}\cc$.

\item For any maximal orthogonal subset $\{ e_1, \ldots, e_n \}$ of $X$, we have $\phi(X) \subseteq \linebreak \{ \phi(e_1), \ldots, \phi(e_n) \}\cc$.

\end{enumerate}
\end{lemma}

\begin{proof}
(1) $\Rightarrow$ (2): Let $\phi$ be normal and let $\{ e_1, \ldots, e_k \} \subseteq X$ be orthogonal. Then $\bar\phi$ maps the top element of ${\mathcal B}(e_1, \ldots, e_k)$ to the top element of ${\mathcal B}(\phi(e_1), \ldots, \phi(e_k))$, that is, $\bar\phi(\{ e_1, \ldots, e_k \}\cc) = \{ \phi(e_1), \ldots, \phi(e_k) \}\cc$.

(2) $\Rightarrow$ (1): Let (2) hold and let $\{ e_1, \ldots, e_k \} \subseteq X$ be orthogonal. Recall that the Boolean algebra ${\mathcal B}(e_1, \ldots, e_k)$ consists of the elements $A\cc \in {\mathcal C}(X, \perp)$, where $A \subseteq \{e_1, \ldots e_k\}$. By assumption, $\bar\phi(A\cc) = (\phi(A))\cc$. Thus $\bar\phi$ establishes an isomorphism between ${\mathcal B}(e_1, \ldots, e_k)$ and ${\mathcal B}(\phi(e_1), \ldots, \phi(e_k))$.

The equivalence of (2) and (3) as well as the equivalence of (4) and (5) are clear. Moreover, (2) clearly implies (4). We conclude the proof by showing that (5) implies (3).

Assume that (5) holds. Let $\{ e_1, \ldots, e_k \}$ be an orthogonal subset of $X$. We extend it to a maximal orthogonal set $E = \{ e_1, \ldots, e_k, e_{k+1}, \ldots, e_m \}$. Furthermore, $f_1 = \phi(e_1), \ldots, f_m = \phi(e_m)$ are pairwise orthogonal elements of $Y$. We extend $\phi(E)$ to a maximal orthogonal subset $F$ of $Y$.

Let $A \subseteq E$. We shall show that $\phi(A\cc) \subseteq \phi(A)\cc$, then in particular (3) will follow. As $\phi$ is orthogonality-preserving, $\phi(A\cc) \perp \phi(E \setminus A)\cc$. Moreover, by assumption, $\phi(A\cc) \subseteq \phi(X) \subseteq \phi(E)\cc \perp (F \setminus \phi(E))\cc$. It follows $\phi(A\cc) \perp \phi(E \setminus A)\cc \vee (F \setminus \phi(E))\cc = (F \setminus \phi(A))\cc$ and hence, by the normality of $Y$, $\phi(A\cc) \subseteq (F \setminus \phi(A))\c = \phi(A)\cc$.
\end{proof}

We observe that normal homomorphisms are, in a restricted sense, linearity-pre\-serv\-ing.

\begin{lemma} \label{lem:normal-implies-line-preserving}
Let $X$ and $Y$ be normal orthogonality spaces and let $\phi \colon X \to Y$ be a normal homomorphism. Let $e, f \in X$ be such that $e \perp f$. If $g \in \{ e, f \}\cc$, then $\phi(g) \in \{ \phi(e), \phi(f) \}\cc$.
\end{lemma}

\begin{proof}
The assertion holds by Lemma~\ref{lem:normal}, property (3).
\end{proof}

An {\it automorphism} of an orthogonality space $(X, \perp)$ is a bijection $\phi \colon X \to X$ such that, for any $e, f \in X$, $e \perp f$ if and only if $\phi(e) \perp \phi(f)$. Automorphisms are always normal homomorphisms, in particular the identity is normal.

\begin{lemma} \label{lem:automorphisms-are-normal}
Let $X$ be a normal orthogonality space and let $\phi \colon X \to X$ be an automorphism. Then $\phi$ is normal.
\end{lemma}

\begin{proof}
$\bar\phi$ is an automorphism of ${\mathcal C}(X,\perp)$.
\end{proof}

We see next that normal homomorphisms are closed under composition.

\begin{lemma}
Let $X$, $Y$, and $Z$ be normal orthogonality spaces and let $\phi \colon X \to Y$ and $\psi \colon Y \to Z$ be normal homomorphisms. Then also $\psi \circ \phi$ is a normal homomorphism.
\end{lemma}

\begin{proof}
Clearly, $\psi \circ \phi$ is orthogonality-preserving. Moreover, the normality follows by means of property (3) in Lemma~\ref{lem:normal}.
\end{proof}

We define the category \NOS\ to consist of the normal orthogonality spaces (of finite rank) and the normal homomorphisms.

We first check whether an inclusion map between normal orthogonality spaces is normal. The following example shows that this is not in general the case.

\begin{example}\label{ex:inclusionnotnormal}
Consider again the orthogonality space $(X, \perp)$ from Example~\ref{Finite-example-2}, which is normal but not Dacey, and let $A = \{ a, e \}$. Then $A \in {\mathcal C}(X, \perp)$ and $(A, \emptyset)$ is a subspace of $(X,  \perp)$, which is normal. Let now $i_A \colon A \to X$ be the inclusion map. We have that  $\{ a \}$ is a maximal orthogonal subset of $A$ and
\[ i_A(A)\cc \;=\; \{a,e\}\cc \;=\; \{a,e\} \;\neq\;  \{ a \} \;=\; \{ a \}\cc \;=\; \{ i_A(a) \}\cc. \]
Hence, by Lemma~\ref{lem:normal}, property {\rm (4)}, $i_A$ is not normal.
\end{example}

\begin{theorem}
Let $(X,\perp)$ be a normal orthogonality space. Then $X$ is a Dacey space if and only if, for any $A \in {\mathcal C}(X, \perp)$, the subspace $(A, \perp)$ is normal and the inclusion map $\iota \colon A \to X$ is a morphism in \NOS.
\end{theorem}

\begin{proof}
Assume first that $X$ is a Dacey space. Let $A \in {\mathcal C}(X, \perp)$. By Lemma~\ref{lem:Dacey} and Proposition~\ref{prop:normal-subspaces}, $(A, \perp)$ is a normal subspace. Moreover, the inclusion map $\iota \colon A \to X \komma x \mapsto x$ is clearly orthogonality-preserving. Let $\{ e_1, \ldots, e_n \}$ be a maximal orthogonal subset of $A$. Then $A = \{e_1, \ldots, e_n\}\cc$ by Lemma~\ref{lem:Dacey}. By Lemma~\ref{lem:normal}, property (4), we conclude that $\iota$ is actually a normal homomorphism.

Conversely, assume that, for any $A \in {\mathcal C}(X, \perp)$, $(A, \perp)$ is normal and the inclusion map $\iota \colon A \to X \komma x \mapsto x$ is a morphism of \NOS. Let $\{ e_1, \ldots, e_n \}$ be a maximal orthogonal subset of some $A \in {\mathcal C}(X, \perp)$. Then again by Lemma~\ref{lem:normal}, property (4), we have that $\iota(A)\cc = \{ \iota(e_1), \ldots, \iota(e_n) \}\cc$, that is, $A = \{ e_1, \ldots, e_n \}\cc$. By Lemma~\ref{lem:Dacey}, we conclude that $X$ is a Dacey space.
\end{proof}

We note that for a normal orthogonality space to be a Dacey space, it is not enough to assume that all subspaces are normal. Indeed, Example~\ref{Finite-example-2} provides a counterexample.

\section{Hermitian spaces}
\label{sec:Hermitian-spaces}

We turn now our attention to orthogonality spaces arising from inner-product spaces. In this section, we compile the necessary background material.

We first consider linear spaces without any additional structure. We are interested in the representation of maps between projective spaces that preserve the collinearity of point triples. The most general results in this area are, to our knowledge, due to Faure \cite{Fau}. Here, we will follow the work of Machala \cite{Mach}. The reader is referred to either of these papers for more detailed information.

By an {\it sfield}, we mean a skew field (i.e., a division ring). Let $V$ be a linear space over an sfield $K$. We write $V\withoutzero = V \setminus \{0\}$ and in accordance with Example~\ref{ex:standard-example-1}, we define $P(V) = \{ \lin x \colon x \in V\withoutzero \}$ to be the projective space associated with $V$. For $x, y, z \in V\withoutzero$, we write $\ell(\lin x, \lin y, \lin z)$ if $\lin x, \lin y, \lin z$ are on a line of $P(V)$, that is, if $x, y, z$ are linearly dependent.

Let $V$ and $V'$ be linear spaces over the sfields $K$ and $K'$, respectively. We call a map $\phi \colon P(V) \to P(V')$ a {\it lineation} if:
\begin{itemize}

\item[\rm (L1)] For any $x, y, z \in V\withoutzero$, $\ell(\lin x, \lin y, \lin z)$ implies $\ell(\phi(\lin x), \phi(\lin y), \phi(\lin z))$.

\end{itemize}
Thus a lineation is a map between projective spaces that preserves the collinearity of point triples. Obviously, (L1) is equivalent to:
\begin{itemize}
\item[\rm (L1')] For any $x, y, z \in V\withoutzero$ such that $\phi(\lin x) \neq \phi(\lin y)$ and $\lin{z} \subseteq \lin{x} + \lin{y}$, we have $\phi(\lin z) \subseteq \phi(\lin x) + \phi(\lin y)$.
\end{itemize}
Thus a lineation can also be understood as follows: if the point $\lin z$ lies on the line through $\lin x$ and $\lin y$ and if $\lin x$ and $\lin y$ are not mapped to the same point, then $\phi(\lin z)$ is on the line through $\phi(\lin x)$ and $\phi(\lin y)$. It is natural to ask whether a lineation is induced by a suitable map between the underlying linear spaces.

Let $K$ be an sfield. We denote by $K\msg = K \setminus \{0\}$ the multiplicative group of $K$. A {\it valuation ring} $F_K$ of $K$ is a subring of $K$ such that, for each $\alpha \in K\msg$, at least one of $\alpha$ or $\alpha^{-1}$ is in $F_K$. In this case, the subgroup $U(F_K)$ of $K\msg$ consisting of the units of $F_K$ is called the {\it group of valuation units}. Obviously, $F_K$ is a local ring, $I_K = F_K \setminus U(F_K) = \{ \alpha \in F_K \colon \alpha = 0 \text{ or } \alpha^{-1} \notin F_K\}$ being its unique maximal left (right) ideal. Let $K'$ be a further sfield; then a ring homomorphism $\rho \colon F_K \to K'$ with kernel $I_K$ is called a {\it place} from $K$ to $K'$. Note that in this case, $\rho$ induces an embedding of the sfield $F_K/I_K$ into $K'$.

Let $V$ be a linear space over an sfield $K$. Let $F_K$ be a valuation ring of $K$ and let $F_V$ be a submodule of $V$ over $F_K$ such that any one-dimensional subspace of $V$ contains a non-zero element of $F_V$. Let $V'$ be a further linear space over an sfield $K'$. Let $\rho \colon F_K \to K'$ be a place from $K$ to $K'$ and let $A \colon F_V \to V'$ be such that (i)~any one-dimensional subspace of $V$ contains a vector in $F_V$ that $A$ does not map to $0$, (ii)~$A$ is additive, and (iii)~for any $x \in F_V$ and $\alpha \in F_K$, we have $A(\alpha x) = \rho(\alpha) A(x)$. Then $A$ is called a {\it generalised semilinear map} from $V$ to $V'$ w.r.t.\ the place $\rho$.

\begin{theorem} \label{thm:generalised-semilinear-map-induces-lineation}
Let $A \colon F_V \to V'$ be a generalised semilinear map between the linear spaces $V$ and $V'$. Then the prescription
\[ \phi_A \colon P(V) \to P(V') \komma \lin{x} \mapsto \lin{A(x)}, \quad\text{where $x \in F_V$ and $A(x) \neq 0$,} \]
defines a lineation.
\end{theorem}

\begin{proof}[Sketch of proof; for full details see the proof of {\rm \cite[Satz 5]{Mach}}]
Each one-di\-men\-sion\-al subspace of $V$ contains by assumption an element $x \in F_V$ such that $A(x) \neq 0$. Moreover, let $y \in \lin x \cap F_V$ such that $A(y) \neq 0$. Then either $y = \alpha x$ or $x = \alpha y$ for some $\alpha \in F_K \setminus \{0\}$. In the former case, we have $A(y) = \rho(\alpha) A(x)$; in the latter case, we have $A(x) = \rho(\alpha) A(y)$. Here, $\rho$ is the place associated with $A$. It follows that $\lin{A(x)} = \lin{A(y)}$. We conclude that we can define $\phi_A$ as indicated.

Let now $x, y, z \in F_V$ such that $A(x), A(y), A(z) \neq 0$ and $\ell(\lin x, \lin y, \lin z)$. We have to show that $\ell(\phi_A(\lin x), \phi_A(\lin y), \phi_A(\lin z))$. We may assume that $\phi_A(\lin x)$, $\phi_A(\lin y)$, and $\phi_A(\lin z)$ are mutually distinct. Let $\alpha, \beta \in K$ be such that $z = \alpha x + \beta y$. Then $\alpha, \beta \neq 0$. Moreover, either $\alpha^{-1} \beta \in F_K$ or $\beta^{-1} \alpha \in F_K$. In the former case, we set $z' = \alpha^{-1} z = x + \alpha^{-1} \beta \, y$; then $z' \in F_V$ and $A(z') = A(x) + \rho(\alpha^{-1} \beta) A(y) \neq 0$ because $\phi_A(\lin x) = \lin{A(x)}$ and $\phi_A(\lin y) = \lin{A(y)}$ are distinct. Hence $\phi_A(\lin z) = \lin{A(z')} = \lin{A(x) + \rho(\alpha^{-1} \beta) A(y)} \subseteq \lin{A(x)} + \lin{A(y)} = \phi_A(\lin x) + \phi_A(\lin y)$ and the assertion follows. In the latter case, we set $z' = \beta^{-1} z$ and proceed similarly.
\end{proof}

Let $A \colon F_V \to V'$ be a generalised semilinear map between the linear spaces $V$ and $V'$. We will then write $I_V = \{ x \in F_V \colon A(x) = 0 \}$. Note that, for any $x \in F_V \setminus I_V$, we have $\lin x \cap F_V = F_K \cdot x$. Moreover, let $\alpha \in F_K$; then $\alpha x \in I_V$ if and only if $\alpha \in I_K$, or in other words, $\alpha x \in F_V \setminus I_V$ if and only if $\alpha \in U(F_K)$.

For a converse of Theorem~\ref{thm:generalised-semilinear-map-induces-lineation}, we need to take into account additional conditions. A lineation $\phi \colon P(V) \to P(V')$ is called {\it non-degenerate} if the following conditions hold:
\begin{itemize}
\item[\rm (L2)] For any linearly independent vectors $x, y \in V\withoutzero$, $\{ \phi(\lin z) \colon z \neq 0, \; z \in \lin{x,y} \}$ contains at least three elements.

\item[\rm (L3)] The image of $\phi$ is not contained in a $2$-dimensional subspace of $V'$.
\end{itemize}

We arrive at the main theorem of \cite{Mach}.

\begin{theorem} \label{thm:Faure}
Every non-degenerate lineation between projective spaces is, in the sense of Theorem~\ref{thm:generalised-semilinear-map-induces-lineation}, induced by a generalised semilinear map.
\end{theorem}

We next consider linear spaces that are equipped with an inner product. For, the maps between projective spaces that we are going to study are supposed to respect first and foremost an orthogonality relation.

A {\it $\star$-sfield} is an sfield equipped with an involutorial antiautomorphism $^\star$. An {\it (anisotropic) Hermitian space} is a linear space $H$ over a $\star$-sfield $K$ that is equipped with an anisotropic, symmetric sesquilinear form $\herm{\cdot}{\cdot} \colon H \times H \to K$. For $x, y \in H$, we write $x \perp y$ if $\herm x y = 0$, and for $x, y \in H\withoutzero$, $\lin x \perp \lin y$ means $x \perp y$.

Let $H$ and $H'$ be Hermitian spaces over the $\star$-sfield $K$ and $K'$, respectively. We call $U \colon H \to H'$ a {\it generalised semiunitary} map if $U$ is a generalised semilinear map w.r.t.\ some place $\rho$ from $K$ to $K'$ and there are a $\lambda \in K$ and a $\lambda' \in K'$ such that
\[ \herm{U(x)}{U(y)} \;=\; \rho(\herm x y \lambda) \, \lambda' \]
for any $x, y \in F_H$. The question arises whether orthogonality-preserving lineations are induced by semiunitary maps. Under particular circumstances, we can give an affirmative answer.

\begin{theorem} \label{thm:generalised-semiunitary-maps}
Let $H$ and $H'$ be finite-dimensional Hermitian spaces over the $\star$-fields $K$ and $K'$, respectively. Assume that $H$ possesses an orthogonal basis consisting of vectors of equal length. Then any non-degenerate orthogonality-preserving lineation $\phi \colon P(H) \to P(H')$ is induced by a generalised semiunitary map.
\end{theorem}

\begin{proof}
By Theorem~\ref{thm:Faure}, $\phi$ is induced by a generalised semilinear map $U \colon F_H \to H'$ w.r.t.\ a place $\rho \colon F_K \to K'$.

We proceed by showing several auxiliary lemmas.

(a) For any $a, b \in F_H \setminus I_H$, $a \perp b$ implies $U(a) \perp U(b)$.

Proof of (a): Assume $a \perp b$. Then $\lin a \perp \lin b$ and hence $\lin{U(a)} = \phi(\lin a) \perp \phi(\lin b) = \lin{U(b)}$ as $\phi$ is orthogonality-preserving. It follows $U(a) \perp U(b)$.

(b) There is an orthogonal basis $b_1, \ldots, b_n \in F_H \setminus I_H$ of $H$ consisting of vectors of equal length.

Proof of (b): By assumption, $H$ possesses an orthogonal basis $b_1, \ldots, b_n$ consisting of vectors of equal length. In view of condition (i) of the definition of a generalised semilinear map, we may assume that $b_1 \in F_H \setminus I_H$. Let $2 \leq i \leq n$; we claim that $b_i \in F_H \setminus I_H$ as well. Assume that $b_i \in I_H$. Then $U(b_1+b_i) = U(b_1-b_i) = U(b_1) \neq 0$ but $b_1 + b_i \perp b_1 - b_i$, in contradiction to (a). Assume that $b_i \notin F_H$. Let $\lambda \in K$ be such that $\lambda b_i \in F_H \setminus I_H$. Then $\lambda^{-1} \in F_K$ would imply that $b_i = \lambda^{-1} \cdot \lambda b_i \in F_H$ contrary to the assumption; hence $\lambda \in I_K$. It follows $\lambda b_1, \lambda b_i \in F_H$ and $U(\lambda b_i + \lambda b_1) = U(\lambda b_i - \lambda b_1) = U(\lambda b_i) \neq 0$ but $\lambda b_i + \lambda b_1 \perp \lambda b_i - \lambda b_1$, again a contradiction to (a).

For the rest of the proof, we fix a basis $b_1, \ldots, b_n$ of $H$ as specified in (b).

(c) $U(b_1), \ldots, U(b_n)$ are vectors of equal length.

Proof of (c): Let $2 \leq i \leq n$. From $b_1 + b_i \perp b_1 - b_i$ it follows by (a) that $U(b_1 + b_i) \perp U(b_1 - b_i)$, that is, $\herm{U(b_1) + U(b_i)}{U(b_1) - U(b_i)} = 0$. By (a), it follows $\herm{U(b_1)}{U(b_1)} = \herm{U(b_i)}{U(b_i)}$. The assertion is shown.

(d) $F_K$ and $I_K$ are closed under $^\star$. Moreover, for any $\alpha \in F_K$, we have $\rho(\alpha^\star) = \rho(\alpha)^\star$.

Proof of (d): Let $\alpha \in I_K \setminus \{0\}$. Assume that $\alpha^\star \notin F_K$. Then $(\alpha^{-1})^\star = (\alpha^\star)^{-1} \in I_K$. Moreover, $\alpha b_1 - b_2$ and $(\alpha^{-1})^\star b_1 + b_2$ are orthogonal vectors in $F_H \setminus I_H$ and hence $-U(b_2) = U(\alpha b_1 - b_2) \perp U((\alpha^{-1})^\star b_1 + b_2) = U(b_2)$, a contradiction. We conclude that $\alpha^\star \in F_K$. Furthermore, $\alpha b_1 - b_2$ and $b_1 + \alpha^\star b_2$ are orthogonal vectors in $F_H \setminus I_H$ and hence $U(b_2) = -U(\alpha b_1 - b_2) \perp U(b_1 + \alpha^\star b_2) = U(b_1) + \rho(\alpha^\star) U(b_2)$. We conclude that $\rho(\alpha^\star) = 0$, that is, $\alpha^\star \in I_K$. We have shown that $I_K$ is closed under $^\star$.

Let now $\alpha \in K \setminus \{0\}$ be such that $\alpha^\star \notin F_K$. Then $(\alpha^\star)^{-1} \in I_K$ and hence also $\alpha^{-1} \in I_K$. This means $\alpha \notin F_K$. It follows that also $F_K$ is closed under $^\star$.

Finally, let $\alpha \in F_K$. We have that $\alpha b_1 - b_2$ and $b_1 + \alpha^\star b_2$ are orthogonal vectors in $F_H \setminus I_H$. It follows that $\rho(\alpha) U(b_1) - U(b_2) \perp U(b_1) + \rho(\alpha^\star) U(b_2)$, that is,
\[ \rho(\alpha) \herm{U(b_1)}{U(b_1)} - \rho(\alpha^\star)^\star \herm{U(b_2)}{U(b_2)} = 0. \]
Thus, by (c), the assertion follows.

(e) Let $\alpha_1, \ldots, \alpha_n \in K$. Then there is an $\alpha \in K \setminus \{0\}$ such that $\alpha^{-1} \alpha_1, \ldots, \alpha^{-1} \alpha_n \in F_K$ and $\alpha^{-1} \alpha_i \notin I_K$ for at least one $i$.

The proof of (e) can be found in \cite[Lemma 6]{Rad}.

(f) Let $x = \alpha_1 b_1 + \ldots + \alpha_n b_n \in F_H$. Then $\alpha_1, \ldots, \alpha_n \in F_K$.

Proof of (f): Assume to the contrary that one of the coefficients is not in $F_K$. By (e), there is an $\alpha \in K$ such that $\alpha^{-1} \alpha_1, \ldots, \alpha^{-1} \alpha_n \in F_K$ and $\alpha^{-1} \alpha_i \notin I_K$ for some $i$. Then $\alpha \notin F_K$ and hence $\alpha^{-1} \in I_K$. Hence $0 = \rho(\alpha^{-1}) U(x) = U(\alpha^{-1} x) = \rho(\alpha^{-1}\alpha_1) U(b_1) + \ldots + \rho(\alpha^{-1}\alpha_n) U(b_n) \neq 0$, because $\rho(\alpha^{-1}\alpha_i) U(b_i) \neq 0$ and the summed vectors are mutually orthogonal. The assertion follows.

Let now $x = \alpha_1 b_1 + \ldots + \alpha_n b_n$ and $y = \beta_1 b_1 + \ldots + \beta_n b_n$ be elements of $F_H$. By (f), $\alpha_1, \ldots, \alpha_n, \beta_1, \ldots, \beta_n \in F_K$. Using (c) and (d), we get
\[ \begin{split}
& \herm{U(x)}{U(y)} \;=\; \rho(\alpha_1) \herm{U(b_1)}{U(b_1)} \rho(\beta_1)^\star + \ldots +
  \rho(\alpha_n) \herm{U(b_n)}{U(b_n)} \rho(\beta_n)^\star \\
& \;=\; \rho(\alpha_1 \beta_1^\star + \ldots + \alpha_n \beta_n^\star) \herm{U(b_1)}{U(b_1)} \\
& \;=\; \rho\left((\alpha_1 \herm{b_1}{b_1} \beta_1^\star + \ldots +
  \alpha_n \herm{b_n}{b_n} \beta_n^\star) \herm{b_1}{b_1}^{-1}\right) \herm{U(b_1)}{U(b_1)} \\
& \;=\; \rho(\herm{x}{y} \herm{b_1}{b_1}^{-1}) \herm{U(b_1)}{U(b_1)},
\end{split} \]
thus the theorem is proved.
\end{proof}

We may observe that the requirement in Theorem~\ref{thm:generalised-semiunitary-maps} regarding the existence of basis vectors of equal length is related to the orderability of the scalar $\star$-field. In the remainder of this section, we review orders on fields and $\star$-fields and we will indicate examples for our above results. For further information, we refer the reader, e.g., to~\cite{Fuc,Pre,Hol}.

We recall that an {\it ordered field} is a field $K$ equipped with a linear order such that the additive reduct becomes a linearly ordered group and the positive elements are closed under multiplication \cite[Chapter 8]{Fuc}. In this case, $K^n$, $n \geq 1$, endowed with the standard inner product,
\[ \hermbig{\scriptsize \begin{pmatrix} \alpha_1 \\ \vdots \\ \alpha_n \end{pmatrix}}%
{\scriptsize \begin{pmatrix} \beta_1 \\ \vdots \\ \beta_n \end{pmatrix}}
\;=\; \alpha_1 \beta_1 + \ldots + \alpha_n \beta_n, \]
is an $n$-dimensional Hermitian space. Moreover, $K^n$ is positive definite and possesses an orthonormal basis.

The elements of $K$ may be roughly classified by means of the linear order. We call
\[ F_K \;=\;
  \{ \alpha \in K \colon |\alpha| \leq n \text{ \rm for some } n \in \Naturals \} \]
the set of {\it finite} elements, which we partition in turn into the two sets
\begin{align*}
M_K \;=\;
& \{ \alpha \in K \colon \tfrac 1 n \leq |\alpha| \leq n \text{ \rm for some } n \in \Naturals \}, \\
I_K \;=\;
& \{ \alpha \in K \colon |\alpha| \leq \tfrac 1 n \text{ \rm for all } n \in \Naturals \},
\end{align*}
consisting of the {\it medial} and the {\it infinitesimal} elements, respectively. Then $F_K$ is a valuation ring of $K$. We have that $U(F_K) = M_K$ is the group of valuation units and, in accordance with the notation used above, the unique maximal ideal of $F_K$ is $I_K = F_K \setminus M_K$.

If $0$ is the only infinitesimal element, that is, if all non-zero elements are medial, then $K$ is called  {\it Archimedean}. This is the case if and only if $K$ is isomorphic to an ordered subfield of $\Reals$; see, e.g., \cite{Fuc}.

Any positive definite Hermitian space over a non-Archimedean ordered field $K$ that possesses an orthonormal basis gives rise to an example for Theorem~\ref{thm:generalised-semiunitary-maps}. Namely, the quotient of the module over $F_K$ of the vectors of finite length modulo the submodule of the vectors of infinitesimal length results in a Hermitian space over $\Reals$. For the details, see, e.g., \cite{Hol}; let us here consider the particular case of non-standard reals.

\begin{example} \label{ex:non-standard-reals}
Let $\Reals^\star$ be the ordered field of hyperreal numbers. The set $F_{\Reals^\star}$ of finite hyperreals is a valuation ring. Moreover, $F_{\Reals^\star} / I_{\Reals^\star}$ is isomorphic to $\Reals$, hence the quotient map induces a place $\rho \colon F_{\Reals^\star} \to \Reals$.

Consider now the linear space $(\Reals^\star)^n$, $n \geq 3$, endowed with the standard inner product. Then
\[ F_{(\Reals^\star)^n} = \{ x \in (\Reals^\star)^n \colon \herm x x \in F_{\Reals^\star} \} = \left\{ {\scriptsize \begin{pmatrix} \alpha_1 \\ \vdots \\ \alpha_n \end{pmatrix}} \colon \alpha_1, \ldots, \alpha_n \in F_{\Reals^\star} \right\} \]
is a submodule of $(\Reals^\star)^n$ over $F_{\Reals^\star}$. Furthermore,
\[ A \colon F_{(\Reals^\star)^n} \to \Reals^n \komma
{\scriptsize \begin{pmatrix} \alpha_1 \\ \vdots \\ \alpha_n \end{pmatrix}} \mapsto
{\scriptsize \begin{pmatrix} \rho(\alpha_1) \\ \vdots \\ \rho(\alpha_n) \end{pmatrix}} \]
is a generalised semiunitary map from $(\Reals^\star)^n$ to $\Reals^n$, inducing a non-degenerate orthog\-onality-preserving lineation.
\end{example}

Consider finally the case that the field is equipped with an involution. Let $K$ be a $\star$-field and let $S_K = \{ \alpha \in K \colon \alpha^\star = \alpha \}$ be its fixed field. Assume that $S_K$, as an additive group, is linearly ordered such that $1 \geq 0$ and, for any $\beta \in K$, $\alpha \geq 0$ implies $\beta \alpha \beta^\star \geq 0$. Then $K$ is called {\it Baer ordered}.

All above considerations allow analogues in this broader context. In particular, a positive definite Hermitian space over a Baer ordered $\star$-field allows the construction of a quotient space over $\Reals$ or $\Complexes$; see \cite{Hol}. Example~\ref{ex:non-standard-reals} can be adapted in a way that the reals are replaced with the field of complex numbers.

\section{Linear orthogonality spaces}
\label{sec:Linear-orthogonality-spaces}

The orthogonality spaces to which we turn in this section are directly related to those that arise from Hermitian spaces. We will show that they are Dacey spaces and hence belong to the normal orthogonality spaces.

\begin{definition}
An orthogonality space $(X, \perp)$ is called {\it linear} if, for any two distinct elements $e, f \in X$, there is a third element $g$ such that $\{e,f\}\c = \{e,g\}\c$ and exactly one of $f$ and $g$ is orthogonal to $e$.
\end{definition}

In other words, for $(X, \perp)$ to be linear means that (i) for distinct, non-orthogonal elements $e, f \in X$ there is a $g \perp e$ such that $\{ e, f \}\c = \{ e, g \}\c$ and (ii) for orthogonal elements $e, f \in X$, there is a $g \notperp e$ such that $\{ e, f \}\c = \{ e, g \}\c$. Note that in both cases $g$ is necessarily distinct from $e$ and $f$.

\begin{example} \label{ex:standard-example-3}
Let $H$ be a Hilbert space and let $(P(H), \perp)$ again be the orthogonality space arising from $H$ according to Example~\ref{ex:standard-example-1}. Then we readily check that $(P(H), \perp)$ is linear.
\end{example}

We start with the following observation. We call an orthogonality space $(X, \perp)$ {\it irredundant} if, for any $e, f \in X$, $\{e\}\c = \{f\}\c$ implies $e = f$. Moreover, we call $(X,\perp)$ {\it strongly irredundant} if, for any $e, f \in X$, $\{e\}\c \subseteq \{f\}\c$ implies $e = f$. Obviously, strong irredundancy implies irredundancy. We may express strong irredundancy also closure-theoretically; cf., e.g., \cite{Ern}. Indeed, $(X,\perp)$ is strongly irredundant exactly if the specialisation order associated with the closure operator $\cc$ is the equality.

\begin{lemma} \label{lem:linear-implies-irredundant}
Linear orthogonality spaces are strongly irredundant.
\end{lemma}

\begin{proof}
Let $(X, \perp)$ be a linear orthogonality space.

We first show that $X$ is irredundant. Let $e$ and $f$ be distinct elements of $X$. If $e$ and $f$ are orthogonal, then $f \perp e$ but $e \notperp e$. If not, there is by the linearity some $g \perp e$ such that $\{ e, f \}\c = \{ e, g \}\c$. Then $g \notin \{ e, g \}\c = \{ e, f \}\c$, hence $g \perp e$ but $g \notperp f$. Hence $\{e\}\c \neq \{f\}\c$ in either case.

Let now $e, f \in X$ be such that $\{e\}\c \subseteq \{f\}\c$. We shall show that then actually $\{e\}\c = \{f\}\c$; by irredundancy, it will follow that $X$ is strongly irredundant. Assume to the contrary that $\{e\}\c \subsetneq \{f\}\c$. Then $e \neq f$ and $e \notperp f$. Hence, by the linearity of $X$, there is a $g \perp e$ such that $\{e,g\}\c = \{e,f\}\c$. But this means $g \in \{e,g\}\cc = \{e,f\}\cc = \{e\}\cc$, a contradiction.
\end{proof}

The following correspondence between linear orthogonality spaces and linear spaces was shown in \cite{Vet3}.

\begin{theorem} \label{thm:orthogonality-spaces-by-orthomodular-spaces}
Let $H$ be a Hermitian space of finite dimension $n$. Then $(P(H), \perp)$ is a linear orthogonality space of rank $n$.

Conversely, let $(X, \perp)$ be a linear orthogonality space of finite rank $n \geq 4$. Then there is a $\star$-sfield $K$ and an $n$-dimensional Hermitian space $H$ over $K$ such that $(X,\perp)$ is isomorphic to $(P(H), \perp)$.
\end{theorem}

Clearly, the assumption regarding the rank cannot be omitted in Theorem~\ref{thm:orthogonality-spaces-by-orthomodular-spaces}. For low ranks, linear orthogonality spaces may be of a much different type than those arising from inner-product spaces.

\begin{example} \label{ex:linear-example-mon}
For $n \geq 2$, let $D_n = \{ 0_1, 1_1, \ldots, 0_n, 1_n \}$, endowed with the orthogonality relation such that $0_i$ and $1_i$, for each $i = 1, \ldots, n$, are orthogonal and no further pair. We easily see that $(D_n, \perp)$ is linear. Note that ${\mathcal C}(D_n, \perp)$ is isomorphic to $\text{\rm MO}_n$, the horizontal sum of $n$ four-element Boolean algebras, which is a modular ortholattice with $2n+2$ elements.
\end{example}

Each linear orthogonality space is a Dacey space and hence normal. The exact relationship is as follows.

Here, an orthogonality space $(X, \perp)$ is called {\it irreducible} if $X$ cannot be partitioned into two non-empty subsets $A$ and $B$ such that $e \perp f$ for any $e \in A$ and $f \in B$.

\begin{theorem} \label{thm:linear-vs-Dacey}
An orthogonality space $(X, \perp)$ is linear if and only if $X$ is an irreducible, strongly irredundant Dacey space. In particular, $X$ is in this case normal.
\end{theorem}

\begin{proof}
Let $(X, \perp)$ be linear. By \cite[Theorem 3.7]{Vet3}, ${\mathcal C}(X, \perp)$ is orthomodular, that is, a Dacey space. By Proposition~\ref{prop:if-Dacey-then-normal}, $X$ is hence normal. By Lemma~\ref{lem:linear-implies-irredundant}, $X$ is strongly irredundant. Assume now that $X = A \cup B$, where $A$ and $B$ are disjoint non-empty subsets of $X$ and $e \perp f$ for any $e \in A$ and $f \in B$. By linearity, for any $e \in A$ and $f \in B$, there is a $g \notperp e$ such that $\{ e, f \}\c = \{ e, g \}\c$. Then $g \notin B$ and consequently $g \in A$ and thus $g \perp f$. It follows $\{ g \}\cc \subseteq \{ e, f \}\cc \cap \{f\}\c = \{e\}\cc$ by orthomodularity and hence $f \in \{ f \}\cc \subseteq \{ e, f \}\cc = \{ e \}\cc \vee \{ g \}\cc = \{ e \}\cc$, in contradiction to $f \in \{e\}\c$. We conclude that $X$ is irreducible.

Conversely, let $(X, \perp)$ be an irreducible, strongly irredundant Dacey space. By the strong irredundancy, $\{e\}\cc$ is, for any $e \in X$, an atom of ${\mathcal C}(X, \perp)$ and it follows that $(X, \perp)$ is atomistic. Furthermore, ${\mathcal C}(X, \perp)$ is a complete orthomodular lattice of finite length. It follows that ${\mathcal C}(X, \perp)$ is in fact a modular lattice and hence fulfils the covering property and the exchange property. Moreover, ${\mathcal C}(X, \perp)$ is irreducible. Indeed, if the centre of ${\mathcal C}(X, \perp)$ contained an element $\emptyset \subsetneq A \subsetneq X$, then each atom would be below $A$ or below $A\c$, that is, we would have $X = A \cup A\c$ and $X$ would not be irreducible.

Let $e, f \in X$ be distinct, non-orthogonal elements. Then $\{e\}\cc$ and $\{f\}\cc$ are distinct atoms and hence $\{ e, f \}\cc = \{e\}\cc \vee \{f\}\cc$ covers $\{e\}\cc$. By orthomodularity, there is an element $g \perp e$ such that $\{ e, f \}\cc = \{e\}\cc \vee \{g\}\cc = \{ e, g \}\cc$, that is, $\{ e, f \}\c = \{ e, g \}\c$.

Let $e, f \in X$ be distinct, orthogonal elements. Since ${\mathcal C}(X, \perp)$ is irreducible, the join of $\{ e\}\cc$ and $\{f\}\cc$ contains a third atom, that is, there is a $g \neq e,f$ such that $g \in \{e,f\}\cc$. By the exchange property, it follows $\{ e, f \}\cc = \{ e, g \}\cc$. Thus $\{ e, f \}\cc = \{ e, g \}\cc$, and $g \notperp e$ because otherwise $g = f$. The proof of the linearity of $X$ is complete.
\end{proof}

\begin{example}
We observe from Theorem~\ref{thm:linear-vs-Dacey} that not every Dacey space is linear. The probably simplest counterexample is $(\Two, \neq)$, the orthogonality space consisting of two orthogonal elements, cf.~Example~\ref{ex:n-element-set}. Obviously, $\Two$ is Dacey but not linear. More generally, the same applies, for any $n \geq 2$, to $({\mathbf n}, \neq)$.
\end{example}

We have seen above that subspaces of normal orthogonality spaces are not necessarily normal. In the present context, the situation is different.

\begin{proposition} \label{prop:subspace-linear-OS-is-linear}
Any subspace of a linear orthogonality space is linear.
\end{proposition}

\begin{proof}
Let $(X, \perp)$ be a linear orthogonality space and let $(A, \perpe{A})$ be a subspace of $X$. For $B \subseteq A$, we have $B\ce{A} = B\c \cap A$. By Theorem~\ref{thm:linear-vs-Dacey}, ${\mathcal C}(X, \perp)$ is orthomodular and hence $B\cce{A} = (B\c \cap A)\c \cap A = B\cc$.

Let $e$ and $f$ be distinct elements of $A$. Then there is a $g \in X$ such that $\{e,f\}\c = \{e,g\}\c$ and exactly one $f$ and $g$ is orthogonal to $e$. Since $g \in \{e,g\}\cc = \{e,g\}\cce{A}$, we have that $g \in A$. Furthermore, we have $\{e,f\}\ce{A} = \{e,f\}\c \cap A = \{e,g\}\c \cap A = \{e,g\}\ce{A}$. The linearity of  $(A, \perpe{A})$ is shown.
\end{proof}

\section{The category \LOS\ of linear orthogonality spaces}
\label{sec:LOS}

We denote by \LOS\ the full subcategory of \NOS\ consisting of linear orthogonality spaces.

Our aim is to describe the morphisms in \LOS. We restrict our considerations to orthogonality spaces that arise from Hermitian spaces. In view of Theorem~\ref{thm:orthogonality-spaces-by-orthomodular-spaces}, the results hence apply to linear orthogonality spaces whose rank is at least~$4$.

\begin{theorem} \label{thm:morphisms-in-LOS}
Let $H$ and $H'$ be finite-dimensional Hermitian spaces. Then a map $\phi \colon P(H) \to P(H')$ is a morphism in \LOS\ if and only if $\phi$ is an orthogonality-preserving lineation.
\end{theorem}

\begin{proof}
Let $\phi \colon P(H) \to P(H')$ be a morphism in \LOS. By definition, $\phi$ preserves the orthogonality relation. Let $x, y, z \in H\withoutzero$ be such that $\phi(\lin x) \neq \phi(\lin y)$ and $z \in \lin{x, y}$. Let $y' \perp x$ be such that $\lin{x,y} = \lin{x,y'}$. By Lemma~\ref{lem:normal-implies-line-preserving}, $\phi(\lin y), \phi(\lin z) \subseteq \phi(\lin x) + \phi(\lin{y'})$. By assumption, $\phi(\lin{x}) \neq \phi(\lin{y})$, so that $\phi(\lin z) \subseteq \phi(\lin x) + \phi(\lin{y'}) = \phi(\lin x) + \phi(\lin{y})$. By criterion (L1'), $\phi$ is a lineation.

Conversely, let $\phi$ be an orthogonality-preserving lineation. Then $\phi$ is a homomorphism of orthogonality spaces. Let $x_1, \ldots, x_n$ be an orthogonal basis of $H$. We have to verify that $\phi(\lin x) \in \{ \phi(\lin{x_1}), \ldots, \phi(\lin{x_n}) \}\cc$, that is, $\phi(\lin x) \subseteq \phi(\lin{x_1}) + \ldots + \phi(\lin{x_n})$ for any $x \in H$; then it will follow by Lemma~\ref{lem:normal}, property (5), that $\phi$ is normal and hence a morphism.

Assume that $x \in \lin{x_1, x_2}$. We have that $\phi(\lin{x_1}) \perp \phi(\lin{x_2})$ and hence $\phi(\lin{x_1}) \neq \phi(\lin{x_2})$. As $\phi$ is a lineation, it follows $\phi(\lin x) \subseteq \phi(\lin{x_1}) + \phi(\lin{x_2})$. The assertion follows thus by an inductive argument.
\end{proof}

A morphism of \LOS\ being a lineation, the question seems natural whether it is non-degenerate. We consider the conditions (L2) and (L3), which define non-degeneracy, separately.

The latter condition is automatic, provided that we assume dimensions of at least $3$.

\begin{lemma} \label{lem:morphisms-L3}
Let $H$ and $H'$ be finite-dimensional Hermitian spaces and assume that the dimension of $H$ is at least $3$. Then any morphism in \LOS\ from  $P(H)$ to $P(H')$ is a lineation fulfilling {\rm (L3)}.
\end{lemma}

\begin{proof}
A morphism $\phi \colon P(H) \to P(H')$ in \LOS\ is by Theorem~\ref{thm:morphisms-in-LOS} an orthogonality-preserving lineation. Since $H$ is at least $3$-dimensional, the image of $\phi$ contains three mutually orthogonal elements. It follows that $\phi$ fulfils (L3).
\end{proof}

In the next lemma, $(\Three, \neq)$ is, in accordance with Example~\ref{ex:n-element-set}, the orthogonality space consisting of three mutually orthogonal elements.

\begin{lemma} \label{lem:three-morphism-1}
Let $H$ and $H'$ be Hermitian spaces of finite dimension $\geq 3$ over the $\star$-sfields $K$ and $K'$, respectively. Assume that there is a morphism in \LOS\ from $P(H)$ to $P(H')$ that does not fulfil {\rm (L2)}. Then there is a $3$-dimensional subspace $H_3$ of $H$ and a morphism in \NOS\ from $(P(H_3), \perp)$ to $(\Three, \neq)$.
\end{lemma}

\begin{proof}
For convenience, we will formulate this proof in the language of orthogonality spaces rather than linear spaces.

Let $\phi \colon P(H) \to P(H')$ be a morphism in \LOS\ that violates {\rm (L2)}. By Theorem~\ref{thm:morphisms-in-LOS}, $\phi$ is an orthogonality-preserving lineation. Moreover, there are $e, f \in P(H)$ such that $e \perp f$ and the image of $\{e,f\}\cc$ under $\phi$ contains exactly two elements. Thus $\phi(\{e,f\}\cc) = \{e', f'\}$, where $e' = \phi(e)$ and $f' = \phi(f)$. We choose a $g \in P(H)$ be such that $g \perp e, f$. Then $e'$, $f'$, and $g' = \phi(g)$ are mutually orthogonal.

Let $H_3$ be the $3$-dimensional subspace of $H$ spanned by $e$, $f$, and $g$. Then we have that $\{ e, f, g \}\cc = P(H_3)$, the orthogonality relation being induced by the inner product on $H_3$. Similarly, let $H'_3$ be the subspace of $H$ spanned by $e'$, $f'$, and $g'$, so that $\{ e', f', g' \}\cc = P(H'_3)$. As $\phi$ is normal, we conclude from Lemma~\ref{lem:normal}, property (3), that the image of $P(H_3)$ under $\phi$ is contained in $P(H'_3)$. In other words, $\phi|_{P(H_3)}$ is an orthogonality-preserving lineation from $P(H_3)$ to $P(H'_3)$.

$(\{ e, f, g \}\cc, \perp)$ is, by Proposition~\ref{prop:subspace-linear-OS-is-linear}, a linear orthogonality space. Furthermore, $\{ e', f', g' \}$, together with the orthogonality relation of $P(H'_3)$, is an orthogonality space isomorphic to $(\Three, \neq)$. In particular, $(\{ e', f', g' \}, \perp)$ is normal. Our aim is to show that there is an orthogonality-preserving map $\psi \colon \{ e, f, g \}\cc \to \{ e', f', g' \}$. Then it will follow that $\psi$ is a morphism of \NOS\ and the lemma will be proved. For, such a map is a homomorphism of orthogonality spaces, and since the image of any set of three mutually orthogonal elements in $\{e,f,g\}\cc$ is $\{ e', f', g' \}$, the normality holds by Lemma~\ref{lem:normal}, property (4).

By Lemma~\ref{lem:normal-implies-line-preserving} we observe that, for any $x \in \{e,f\}\cc$ such that $\phi(x) = e'$, we have $\phi(\{g,x\}\cc) \subseteq \{e',g'\}\cc$, and for any $x \in \{e,f\}\cc$ such that $\phi(x) = f'$, we have $\phi(\{g,x\}\cc) \subseteq \{f',g'\}\cc$. Furthermore, for any $y \in \{e,f,g\}\cc$, there is an $x \in \{e,f\}\cc$ such that $y \in \{g,x\}\cc$. We conclude that
\begin{equation} \label{fml:three-morphism-0}
\phi(\{e,f,g\}\cc) \subseteq \{e',g'\}\cc \cup \{f',g'\}\cc.
\end{equation}
We now distinguish three cases.

{\it Case 1.} There is an $h \in \{e,g\}\cc$ such that $h' = \phi(h) \neq e', g'$. Note that $h' \in \{e',g'\}\cc$. We claim that
\begin{equation} \label{fml:three-morphism-1}
\phi(\{f,h\}\cc) \;=\; \{f',h'\}.
\end{equation}
Let $x \in \{f,h\}\cc$. Since $x \neq g$, there is a unique $t \in \{e,f\}\cc$ such that $x \in \{g,t\}\cc$. Depending on whether $\phi(t) = e'$ or $\phi(t) = f'$, we have that $\phi(x) \in \{e',g'\}\cc$ or $\{f',g'\}\cc$. Furthermore, since $\phi(x) \in \{f',h'\}\cc$, we conclude that either $\phi(x) = h'$ or $\phi(x) = f'$. Thus (\ref{fml:three-morphism-1}) is shown.

We next claim that
\begin{equation} \label{fml:three-morphism-2}
\phi(\{e,f,g\}\cc) \subseteq \{e',g'\}\cc \cup \{f'\}.
\end{equation}
Let $x \in \{e,f,g\}\cc$ such that $\phi(x) \notin \{e',g'\}\cc$. Note that then $\phi(x) \in \{ f', g' \}\cc$ by (\ref{fml:three-morphism-0}). Moreover, there is a unique $y \in \{f,h\}\cc \cap \{e,x\}\cc$. Then $x \in \{ e, y \}\cc$ and, by (\ref{fml:three-morphism-1}), either $\phi(y) = h'$ or $\phi(y) = f'$. If $\phi(y) = h'$, then $\phi(x) \in \{e',h'\}\cc = \{e',g'\}\cc$, in contradiction to our assumption. Hence we have $\phi(y) = f'$ and $\phi(x) \in \{e',f'\}\cc$ and we conclude $\phi(x) = f'$. Thus (\ref{fml:three-morphism-2}) is shown.

Finally, let $\tau \colon \{e',g'\}\cc \to \{e',g'\}$ be any orthogonality-preserving map. We define
\[ \psi \colon \{ e, f, g \}\cc \to \{ e', f', g' \} \komma x \mapsto
\begin{cases} \tau(\phi(x)) & \text{if $\phi(x) \in \{e',g'\}\cc$,} \\
                         f' & \text{if $\phi(x) = f'$.}
\end{cases} \]
Then $\psi$ is orthogonality-preserving, as desired.

{\it Case 2.} There is a $h \in \{f,g\}\cc$ such that $\phi(h) \neq f', g'$. Then we argue similarly to Case 1.

{\it Case 3.} $\phi(\{ e, g \}\cc) = \{ e', g' \}$ and $\phi(\{ f, g \}\cc) = \{ f', g' \}$. Let then $x \in \{e,f,g\}\cc$ such that $x \notin \{e,g\}\cc \cup \{f,g\}\cc$. By (\ref{fml:three-morphism-0}), $\phi(x) \in \{e',g'\}\cc$ or $\phi(x) \in \{f',g'\}\cc$. In the former case, let $y \in \{e,g\}\cc$ be such that $x \in \{f,y\}\cc$; then $\phi(y) = e'$ and hence also $\phi(x) = e'$, or $\phi(y) = g'$ and hence $\phi(x) = g'$. In the latter case, we similarly see that $\phi(x) = f'$ or $\phi(x) = g'$.

We conclude that $\phi(\{e,f,g\}\cc) = \{ e', f', g' \}$. Taking $\psi = \phi$, we again have that $\psi$ is orthogonality-preserving.
\end{proof}

We may formulate the conclusion of Lemma~\ref{lem:three-morphism-1} measure-theoretically. By a {\it measure} on a finite-dimensional Hermitian space $H$, we mean a map $\mu$ from ${\mathcal C}(H)$ to the real unit interval such that (i) $\mu(A + B) = \mu(A) + \mu(B)$ for any orthogonal subspaces $A$ and $B$ of $H$ and (ii) $\mu(H) = 1$. We call a measure {\it two-valued} if its image consists of $0$ and $1$ only.

\begin{lemma} \label{lem:three-morphism-2}
Let $H$ and $H'$ be Hermitian spaces of finite dimension $\geq 3$ over the $\star$-sfields $K$ and $K'$, respectively. Assume that there is a morphism in \LOS\ from $P(H)$ to $P(H')$ that does not fulfil {\rm (L2)}. Then there is a $3$-dimensional Hermitian space over $K$ that possesses a two-valued measure.
\end{lemma}

\begin{proof}
By Lemma~\ref{lem:three-morphism-1}, there is a $3$-dimensional subspace $H_3$ of $H$ and a map from $P(H_3)$ to the $3$-element set $\Three$ such that any three mutually orthogonal elements of $P(H_3)$ are mapped to distinct elements. Assigning $1$ to one of the elements of $\Three$ and $0$ to the remaining two, we may construct a two-valued measure as asserted.
\end{proof}

Combining the results that we have achieved so far, we arrive at the following statement.

\begin{proposition} \label{prop:morphisms-of-LOS-and-semiunitary-maps}
Let $H$ and $H'$ be Hermitian spaces of finite dimension $\geq 3$ over $\star$-fields. Assume moreover that {\rm (1)} $H$ possesses a basis of vectors of equal length and {\rm (2)} no $3$-dimensional subspace of $H$ possesses a two-valued measure. Then any morphism in \LOS\ between $P(H)$ and $P(H')$ is induced by a generalised semiunitary map.
\end{proposition}

\begin{proof}
Let $\phi \colon P(H) \to P(H')$ be a morphism in \LOS. By Theorem~\ref{thm:morphisms-in-LOS}, $\phi$ is an orthogonality-preserving lineation and by Lemma~\ref{lem:morphisms-L3}, $\phi$ fulfils (L3).

Assume now that $\phi$ does not fulfil (L2). By Lemma~\ref{lem:three-morphism-2}, $H_3$ possesses a two-valued measure, in contradiction to condition (2). We conclude that $\phi$ does fulfil (L2) and is hence non-degenerate.

Taking into account condition (1), the assertion now follows by Theorem~\ref{thm:generalised-semiunitary-maps}.
\end{proof}

The question arises when a Hermitian space fulfils the conditions of Proposition~\ref{prop:morphisms-of-LOS-and-semiunitary-maps}, which are admittedly rather technical. It would in particular be desirable to find a formulation referring to the scalar fields only. Whereas an answer seems to be difficult to find in general, the situation is somewhat more transparent in the case of Baer ordered $\star$-fields.

\begin{theorem} \label{thm:morphisms-of-LOS-and-semiunitary-maps}
Let $H$ be a positive definite Hermitian space of finite dimension $\geq 3$ over the Baer ordered $\star$-field $K$, and let $H'$ be a further finite-dimensional Hermitian space. Assume moreover that $K$ has the following properties:
\begin{itemize}

\item[\rm (1)] For any $\alpha \in S_K$ such that $\alpha \geq 0$, there is a $\beta \in K$ such that $\alpha = \beta \beta^\star$.

\item[\rm (2)] $K^3$, endowed with the standard inner product, does not possess a two-valued measure.

\end{itemize}
Then any morphism in \LOS\ between $P(H)$ and $P(H')$ is induced by a generalised semiunitary map.
\end{theorem}

\begin{proof}
Let $x \in H\withoutzero$. As $H$ is positive definite, there is, by (1), a $\beta \in K$ such that $\herm{x}{x} = \beta \beta^\star$ and hence $\herm{\frac 1 \beta x}{\frac 1 \beta x} = 1$. That is, every subspace of $H_3$ contains a unit vector and $H_3$ possesses an orthonormal basis.

It furthermore follows that $H_3$ can be identified with $K^3$, endowed with the standard inner product. Thus, by (2), $H_3$ does not possess a two-valued measure. The assertion follows now from Proposition~\ref{prop:morphisms-of-LOS-and-semiunitary-maps}.
\end{proof}

Theorem~\ref{thm:morphisms-of-LOS-and-semiunitary-maps} leads in turn to the question how to characterise the Baer ordered $\star$-fields $K$ fulfilling condition (2), that is, how to exclude the existence of two-valued measures on $K^3$. We will give a sufficient criterion; see Lemma~\ref{lem:Euclidean-fixed-field-and-two-valued-measures} below. We need several preparatory steps, some of which might be interesting in their own right.

We recall that an ordered field is called {\it Euclidean} if any positive element is a square. In the proof of the next lemma, we follow the lines of Piron's proof of Gleason's Theorem \cite[p.~75--78]{Pir}.

\begin{theorem} \label{thm:classical-space-no-two-valued-measures}
A $3$-dimensional positive definite Hermitian space over a Euclidean subfield of $\Reals$ does not possess a two-valued measure.
\end{theorem}

\begin{proof}
Let $R$ be a Euclidean subfield of the reals and let $H$ be a $3$-dimensional positive definite Hermitian space over $R$.

We claim that each one-dimensional subspace of $H$ possesses a unit vector. Indeed, the only automorphism of $R$ is the identity, hence $\star = \id$. Hence the assertion follows like in the proof of Theorem~\ref{thm:morphisms-of-LOS-and-semiunitary-maps}.

Let us assume that there is a two-valued measure $\mu$ on $H$, that is, a map $\mu \colon P(H) \to \{0,1\}$ such that, among any three orthogonal elements $\lin x, \lin y, \lin z \in P(H)$, exactly one is mapped to $1$. Pick $b_3 \in H\withoutzero$ such that $\mu(\lin{b_3}) = 1$ and let $b_1, b_2, b_3$ be an orthogonal basis of $H$. By the previous paragraph, we can suppose that $b_1, b_2, b_3$ are unit vectors. We may hence identify $H$ with $R^3$, endowed with the standard inner product.

We have that $\mu(\lin{\vector{0}{0}{1}}) = 1$ and consequently $\mu(\lin{\vector{\alpha}{\beta}{0}}) = 0$ for any elements $\alpha, \beta \in R$ that are not both $0$. The map $\iota \colon R^2 \to P(R^3) \komma (\alpha,\beta) \mapsto \lin{\vector{\alpha}{\beta}{1}}$ establishes a one-to-one correspondence between $R^2$ and the set of those elements of $P(R^3)$ that are not orthogonal to $b_3$. We shall write $\bar\mu$ for $\mu \circ \iota$. Let $\bar 0$ be the origin of $R^2$; then $\bar\mu(\bar 0) = 1$. We proceed by showing several auxiliary statements.

(a) Let $L \subseteq R^2$ be a line and let $r \in L$ be the element closest to $\bar 0$. Then $\bar\mu(r) \geq \bar\mu(s)$ for any $s \in L$.

Proof of (a): We may assume that $s \neq r$. Let $\lin l \in P(R^3)$ be parallel to $L$. Then $\lin l \perp \iota r$ and $\mu(\lin l) = 0$. Moreover, let $t \in L$ be such that $\iota t \perp \iota s$. Then $\iota r$ and $\lin l$ span the same $2$-dimensional subspace of $R^3$ as $\iota s$ and $\iota t$. Hence $\bar\mu(r) = \bar\mu(r) + \mu(\lin l) = \bar\mu(s) + \bar\mu(t) \geq \bar\mu(s)$.

We will denote by $\norm r$ the (Euclidean) distance between $\bar 0$ and some $r \in R^2$.

(b) For any $r \in R^2$ and $\tau \in R$ such that $0 < \tau \leq 1$, we have $\bar\mu(r) \leq \bar\mu(\tau \cdot r)$.

Proof of (b): Let $r^\perp$ arise from rotating $r$ by $\frac \pi 2$. Consider
\[ s \;=\; \tau \cdot r + \sqrt{\tau(1-\tau)} \cdot r^\perp. \]
Then $\angle \bar 0 \, s \, r = \frac \pi 2$, hence $\bar\mu(r) \leq \bar\mu(s)$ by (a). Likewise, we have $\angle s \, \tau r \, \bar 0 = \frac \pi 2$, hence $\bar\mu(s) \leq \bar\mu(\tau r)$ again by (a).

In what follows, $D_\omega \colon R^2 \to R^2$ denotes the rotation by $\omega$.

(c) Let $r \in R^2$, $\; n \geq 1$, and $s = \cos \frac{\pi}{2^n} D_{\frac{\pi}{2^n}} r$. Then $s \in R^2$ and $\bar\mu(r) \leq \bar\mu(s)$.

Proof of (c): From the fact that, for any $x \in \Reals$, we have $\cos^2 \frac x 2 = \frac 1 2 (1+\cos x)$, we conclude that $\cos \frac{\pi}{2^n}, \sin \frac{\pi}{2^n} \in R$ for all $n$. It follows that $s \in R^2$. Moreover, $\angle r \, s \, \bar 0 = \frac \pi 2$. Hence the last assertion follows from (a).

(d) $\lim_{n \to \infty} \cos^n \frac{\pi}{n} = 1$.

Proof of (d): By L'Hospital's rule, we have $\lim_{x \to 0} \frac{\ln \cos \pi x} x = -\pi \lim_{x \to 0} \tan \pi x = 0$. Hence $\lim_{n \to \infty} \ln \cos^n \frac{\pi}{n} = \lim_{n \to \infty} n \ln \cos \frac{\pi}{n} = 0$ as well and the assertion follows.

(e) Let $r \in R^2$ such that $\norm r > 1$. Then $\bar\mu(r) \leq \bar\mu(-\frac 1 {{\norm r}^2} r)$.

Proof of (e): Because of (d), we may choose an $m$ large enough such that $\cos^{2^m} \!\! \frac{\pi}{2^m} > \frac 1 {{\norm r}^2}$. Let $\omega = \frac{\pi}{2^m}$ and define a sequence $r^{(i)} \in R^2$, $\, i = 0, \ldots, 2^m$, as follows:
\[ r^{(0)} = r, \quad r^{(i+1)} = \cos \omega \cdot D_\omega r^{(i)} \text{ for $0 \leq i < 2^m$}. \]
By (c), $r^{(i)} \in R^2$ for all $i$, and $\bar\mu(r) = \bar\mu(r^{(0)}) \leq \bar\mu(r^{(1)}) \leq \ldots \leq \bar\mu(r^{(2^m)})$. Moreover, $r^{(2^m)} = -\cos^{2^m}\!\omega \cdot r$ and $-\cos^{2^m}\!\omega  < -\frac 1 {{\norm r}^2}$. 
Hence, by (b), it follows $\bar\mu(r^{(2^m)}) \leq \bar\mu(-\frac 1 {{\norm r}^2} r)$.

(f) Let $r \in R^2$ be such that $\norm{r} > 1$. Then $\bar\mu(r) = 0$.

Proof of (f): Assume that $\bar\mu(r) = 1$. By (e), $\bar\mu(r) \leq \bar\mu(- \tfrac 1 {{\norm r}^2} r)$, hence $\bar\mu(- \tfrac 1 {{\norm r}^2} r) = 1$. But $\iota r$ and $\iota(-\tfrac 1 {{\norm r}^2} r)$ are perpendicular, a contradiction.

(g) There are $r, s, t \in R^2$ such that $\iota r, \iota s, \iota t$ are mutually orthogonal and $\norm r, \norm s, \linebreak \norm t > 1$.

Proof of (g): Consider $(2, 0)$, $(-\frac 1 2, 1)$, and $(-\frac 1 2, -\frac 5 4)$.
\end{proof}

Whereas Theorem~\ref{thm:classical-space-no-two-valued-measures} might in the present context be of limited applicability, its following corollary is more useful.

\begin{lemma} \label{lem:Euclidean-subfield-of-reals-in-fixed-field}
Let $R$ be a subfield of the Baer ordered $\star$-field $K$. Assume moreover that $R$ is isomorphic to a Euclidean subfield of $\Reals$. Then $K^3$, endowed with the standard inner product, does not possess a two-valued measure.
\end{lemma}

\begin{proof}
The inclusion map $R^3 \to K^3$ induces a map $P(R^3) \to P(K^3)$, which is injective and orthogonality-preserving. Hence, if there is an orthogonality-preserving map from $P(K^3)$ to $\Three$, there is an orthogonality-preserving map from $P(R^3)$ to $\Three$.

Assume that there is a two-valued measure on $K^3$. Then it follows that there is a two-valued measure on $R^3$. This in turn is impossible by Theorem~\ref{thm:classical-space-no-two-valued-measures}.
\end{proof}

Let $R$ be a subfield of an ordered field $S$. Equipped with the inherited order, $R$ is an ordered field again. We note that we may speak about the infinitesimal and the medial elements of $R$ without the need to specify whether we refer to $R$ or $S$. Indeed, we have that $M_R = M_S \cap R$ and $I_R = I_S \cap R$ because $R$ contains the rational subfield of $S$.

\begin{lemma} \label{lem:quadratic-extension}
Let $S$ be a Euclidean field and let $R$ be an Archimedean subfield of $S$. Then any quadratic extension of $R$ is Archimedean as well.
\end{lemma}

\begin{proof}
Let $\gamma \in R^+$ not possess a square root in $R$. We have to show that $R(\sqrt \gamma)$ is Archimedean. Since $R$ is Archimedean, we have that $R \subseteq M_S \cup \{0\}$. In particular, $\gamma \in M_S$ and this implies that also $\sqrt \gamma \in M_S$.

Furthermore, $R(\sqrt \gamma) = \{ \alpha + \beta \sqrt \gamma \colon \alpha, \beta \in R \}$. Assume that $\alpha, \beta \in R$ are such that $\alpha + \beta \sqrt \gamma$ is a non-zero infinitesimal element. Note that then $\beta \neq 0$. Since $\alpha - \beta \sqrt \gamma \in F_S$, it follows that also $\alpha^2 - \beta^2 \gamma = (\alpha - \beta \sqrt \gamma)(\alpha + \beta \sqrt \gamma)$ is infinitesimal. But the only infinitesimal element of $R$ is $0$, hence $\gamma = (\frac \alpha \beta)^2$, a contradiction. We conclude that $R(\sqrt \gamma) \subseteq M_S \cup \{0\}$ and it follows that $R(\sqrt \gamma)$ is Archimedean.
\end{proof}

\begin{lemma} \label{lem:Euclidean-fixed-field}
Let $S$ be a Euclidean field. Then there is a smallest Euclidean subfield $R$ of $S$. Moreover, $R$ is isomorphic to an ordered subfield of $\Reals$.
\end{lemma}

\begin{proof}
Let $Q$ be the rational subfield of $S$ and let $R$ be the Euclidean closure of $Q$, that is, the smallest subfield of $S$ containing $Q$ and such that $\alpha \in R^+$ implies $\sqrt \alpha \in R$. Any Euclidean subfield of $S$ contains $Q$ and hence $R$, hence $R$ is the smallest Euclidean subfield of $S$. 

$Q$ is isomorphic to $\Rationals$ and since $\Rationals$ can be ordered in only one way, $Q$ is in fact an ordered subfield of $S$ that is order-isomorphic to $\Rationals$ equipped with its natural order. We conclude that $Q$ is Archimedean.

Furthermore, the formation of the Euclidean closure of $Q$ is the result of a double inductive process, each step being a quadratic extension; cf., e.g., \cite[Proposition~2.12]{Lam}. Since by Lemma~\ref{lem:quadratic-extension} the Archimedean property is preserved in each step and since the union of Archimedean subfields is Archimedean again, we conclude that also $R$ is Archimedean and hence an ordered subfield of~$\Reals$.
\end{proof}

Let $K$ be a Baer ordered $\star$-field. Then its fixed field $S_K$ is endowed with a linear order~$\leq$ with the effect that $(S_K; +, 0, \leq)$ is a totally ordered group and, for any $\alpha, \beta \in S_K$, $\alpha \geq 0$ implies $\alpha \beta^2 \geq 0$. We will say that the fixed field of $K$ is Euclidean if, in $S_K$, any positive element is a square. Clearly, $S_K$ is in this case actually an ordered field and this ordered field is Euclidean.

\begin{lemma} \label{lem:Euclidean-fixed-field-and-two-valued-measures}
Let $K$ be a Baer ordered $\star$-field whose fixed field is Euclidean. Then $K^3$, endowed with the standard inner product, does not possess a two-valued measure.
\end{lemma}

\begin{proof}
By Lemma~\ref{lem:Euclidean-fixed-field}, $S_K$ possesses a subfield that is isomorphic to a Euclidean subfield of the reals. Hence the assertion follows from Lemma~\ref{lem:Euclidean-subfield-of-reals-in-fixed-field}.
\end{proof}

We arrive at our main result. We denote by \EOS\ the full subcategory of \LOS, and hence of \NOS, consisting of orthogonality spaces that arise from (finite-dimensional) positive definite Hermitian spaces over Baer ordered $\star$-field whose fixed field is Euclidean.

\begin{theorem}
Let $H$ be a positive definite Hermitian space of finite dimension $\geq 3$ over a Baer ordered $\star$-field with a Euclidean fixed field. Let $H'$ be a further finite-dimensional Hermitian space. Then any morphism in \LOS\ between $P(H)$ and $P(H')$ is induced by a generalised semiunitary map.

In particular, any morphism in \EOS\ between orthogonality spaces of rank $\geq 3$ is induced by a generalised semiunitary map.
\end{theorem}

\begin{proof}
We verify the two conditions in Theorem~\ref{thm:morphisms-of-LOS-and-semiunitary-maps}.

Let $K$ be the scalar $\star$-field of $H$. Since the fixed field of $K$ is Euclidean, condition (1) is fulfilled. Moreover, by Lemma~\ref{lem:Euclidean-fixed-field-and-two-valued-measures}, $K^3$ does not possess two-valued measures. Hence also condition (2) holds.
\end{proof}

\section{Conclusion}
\label{sec:Conclusion}

The objective of this paper has been to establish a categorical framework for orthogonality spaces. The latter structures can be identified with undirected graphs and in the context of graph theory, categories have already been studied, e.g., in \cite{Faw}. However, the categories discussed by the graph theorists have turned out to be unsuitable in the present context. Our primary example originates from quantum physics and hence our intention has been to introduce a category whose morphisms, when applied to linear orthogonality spaces, come close to linear mappings. We have therefore introduced normal orthogonality spaces, which are still more general than linear orthogonality spaces. But normality suggests a definition of morphisms such that, when applied in the context of inner-product spaces, not only the orthogonality relation is taken into account but also the linear structure.

We believe that the presented work is a first step into an area that offers numerous issues for further investigations. For instance, we have shown that the morphism between specific Hermitian spaces can be represented by generalised semiunitary maps. It has remained open whether a similar statement is possible for a broader class. In fact, whereas generalised semilinear maps have been studied by several authors, there does not seem to exist any detailed account on maps also preserving an inner product. Moreover, we have seen that the existence of two-valued measures plays a role in the discussion. This question as well as Gleason's Theorem have been  studied, with some exceptions \cite{Dvu}, in the context of classical fields, whereas the present context suggests to take into account further non-classical fields.

To mention finally a particularly interesting issue, recall that the lattice-theoretic approach has often been criticised for its inability to deal appropriately with common constructions of Hilbert spaces, like direct sums and tensor products. In the framework of orthogonality spaces, the situation is much different and a categorical framework might be useful for these matters.

{\bf Acknowledgement.} The authors acknowledge the support by the bilateral Austrian Science Fund (FWF) project I 4579-N and Czech Science Foundation (GA\v CR) project 20-09869L ``The many facets of orthomodularity''.

\end{document}